\documentclass{article}
\usepackage[english]{babel}
\usepackage[margin=36mm]{geometry}

\usepackage{amsmath}
\usepackage{mathtools}
\usepackage{graphicx}
\usepackage{amssymb}
\usepackage[colorlinks=true, allcolors=blue]{hyperref}
\usepackage{amsthm}
\usepackage{enumitem}
\usepackage{tikz}
\usepackage{tikz-cd}
\usepackage{todonotes}
\usepackage{bbm}

\usetikzlibrary{arrows.meta,
                decorations.markings, 
                 positioning}
\usetikzlibrary{topaths,calc}
\usepackage{pgfplots}
\pgfplotsset{compat=1.18}
\usepackage{adjustbox}
\usepackage{subcaption} 
\usepackage{caption} 
\usepackage{floatrow}
\usepackage{nicematrix}
\usepackage{algpseudocode}
\usepackage{appendix}
\usepackage{comment}

\usepackage{bbm}
\usepackage{bm}

\usepackage[capitalise]{cleveref}
\usepackage{placeins}
\usepackage{cases}

\usepackage{ifthen}

\usepackage{float}

\newcommand{\Keywords}[1]{\vspace{3mm}\par\noindent\hspace*{10mm}
\parbox{140mm}{\small {\textbf {\textit{Key words}}:} \rm #1}\par}

\newcommand{\Email}[1]{\ifthenelse{\equal{#1}{}}{}{\par\noindent {\rm E-mail: }{\it  #1} \par}}
\newcommand{\EmailMarked}[1]{\ifthenelse{\equal{#1}{}}{}{\par\noindent $^*$~{\rm E-mail: }{\it  #1} \par}}
\newcommand{\URLaddress}[1]{\ifthenelse{\equal{#1}{}}{}{\par\noindent {\rm URL: }{\tt  #1} \par}}
\newcommand{\URLaddressMarked}[1]{\ifthenelse{\equal{#1}{}}{}{\par\noindent $^*$~{\rm URL: }{\tt  #1} \par}}
\newcommand{\EmailD}[1]{\ifthenelse{\equal{#1}{}}{}{\par\noindent {$\phantom{^{{\rm a)}}}$~\rm E-mail: }{\it  #1} \par}}
\newcommand{\EmailDD}[1]{\ifthenelse{\equal{#1}{}}{}{\par\noindent {$\phantom{{}^{\dag^1}}$~\rm E-mail: }{\it  #1} \par}}

\newtheorem{theorem}{Theorem}[section]
\newtheorem{lemma}{Lemma}[section]
\newtheorem{proposition}{Proposition}[section]

\newtheorem{definition}{Definition}[section]
\newtheorem{example}{Example}[section]
\newtheorem{remark}{Remark}[section]

\newcommand{\LL}{\mathcal{L}}

\newcommand{\im}{\mathrm{im}}

\newcommand{\tr}{\operatorname{tr}}
\newcommand{\SC}{\operatorname{SC}}
\newcommand{\loops}{\operatorname{loop}}
\newcommand{\art}{\operatorname{art}}
\newcommand{\ic}{\operatorname{chord}_{\operatorname{iso}}}

\newcommand{\Ber}{\operatorname{Bernoulli}}
\newcommand{\up}{\operatorname{up}}
\newcommand{\lap}{\operatorname{Lap}}
\newcommand{\GS}{\operatorname{GS}}

\newcommand{\EE}{\mathbb{E}}

\newcommand{\Var}{\mathbbm{Var}}
\newcommand{\Cov}{\mathbbm{Cov}}
\newcommand{\NCev}{\mathcal{NC}_{\operatorname{even}}}
\newcommand{\NC}{\mathcal{NC}}

\usepackage{authblk}
\date{} 

\title{Eigenvalue Distribution of the Laplacian on Random Complexes}
\author[1]{Rui Dong}

\begin{document}
\maketitle
\Email{\href{mailto:ruidong.research@gmail.com}{ruidong.research@gmail.com}}

\begin{abstract}
We study the empirical eigenvalue distribution of a standardized up Laplacian of the Linial-Meshulam model $Y_{q+1}(n, p)$.
We first give a definition of a Gaussian--semicircle law: $\mathcal{N}(0, \sigma^2)\boxplus \SC(s\sigma^2)$,
and give a combinatorial formula of its moments in terms of pairing partitions.
In addition to that,
we also prove that the limiting empirical eigenvalue distribution of this  standardized up Laplacian follows a Gaussian--semicircle law,
in the sense of almost surely weak convergence.
\end{abstract}

\tableofcontents
\section{Introduction}

Random simplicial complexes provide a natural higher-dimensional generalization of random graphs.
Among the most fundamental models is the Linial--Meshulam random complex,
introduced by Linial and Meshulam \cite{MR2260850} in dimension two and subsequently generalized by Meshulam and Wallach \cite{MR2504405} to arbitrary dimensions.
In the Linial--Meshulam model $Y_{q+1}(n,p)$,
the complete $q$-dimensional skeleton is present and each $(q+1)$-simplex is included independently with probability $p$. 
Since its introduction,
the model has been studied extensively from both topological and spectral perspectives,
with particular attention to homological phase transitions, 
expansion,
and the spectra of higher-dimensional Laplacians.

From the viewpoint of random matrix theory, 
the eigenvalue distribution of random graphs provides an important point of departure. 
For an Erd\H{o}s--R\'enyi graph $G(n,p)$,
the \emph{empirical eigenvalue distribution} (EED) of the centered and scaled adjacency matrix converges to Wigner's semicircle law, 
whereas Ding and Jiang \cite{MR2759729} showed that a suitably normalized EED of the centered Laplacian converges to the free convolution of a Gaussian distribution and the semicircle law. 
The appearance of the Gaussian component in the Laplacian case is naturally associated with fluctuations of the random degrees.
This Gaussian--semicircle free convolution phenomenon suggests a corresponding question for higher-dimensional Laplacians: 
does the centered up-Laplacian of a random simplicial complex exhibit an analogous limiting eigenvalue distribution?

A substantial spectral theory has been developed for the Linial--Meshulam model.
Gundert and Wagner \cite{MR3557457} studied higher-dimensional analogues of the normalized Laplacian and adjacency matrix and established concentration of their eigenvalues in the regime $p=\Omega(\log n/n)$. 
Knowles and Rosenthal \cite{MR3689342} subsequently studied the adjacency matrix of $Y_{q+1}(n,p)$. 
Under the assumption
$$
np(1-p)\gg \log^4 n,
$$
they proved eigenvalue confinement and showed that the EED of the appropriately centered and scaled adjacency matrix converges to the semicircle law. These results establish a close connection between the spectral theory of random simplicial complexes and that of random matrix models with dependent entries.

The EED of higher-dimensional Laplacians has also been considered directly.
Kanazawa \cite{MR4340474} developed a general framework for homogeneous and spatially independent random simplicial complexes and proved convergence of the EEDs of their Laplacians,
with the Linial--Meshulam model appearing as a special case. 
The argument is based on local weak convergence and, 
in the Linial--Meshulam setting, 
is formulated in the sparse regime in which the expected number of $(q+1)$-simplices incident to a fixed $q$-simplex remains of constant order. 
Thus,
although convergence of Laplacian EEDs is known in sparse random-complex regimes, 
this result does not identify the fluctuation spectrum of the centered up-Laplacian in the dense regime with fixed $p\in (0,1)$.

In this paper,
we study precisely this dense-regime fluctuation problem.
Let $L_q^{\mathrm{up}}$ denote the unnormalized up-Laplacian of $Y_{q+1}(n,p)$.
Since
$$
\mathbb{E}L_q^{\mathrm{up}}=pB_{q+1}B_{q+1}^{\top},
$$
and the boundary matrix of the complete complex satisfies
$$
B_{q+1}B_{q+1}^{\top}=nP_q,
$$
where $P_q$ is the orthogonal projection onto $\im(B_{q+1})$,
the natural centered and normalized fluctuation matrix is
$$
W_n
=
\frac{L_q^{\mathrm{up}}-\EE[L_q^{\mathrm{up}}]}{\sqrt n}.
$$
We remove the deterministic zero eigenspace of $W_n$ and consider the EED on the image of $P_q$.

We exhibit two main results in \cref{sec:main_result}.
In \cref{sec:result_moments} we define a Gaussian--semicircle law
\[
\mu_{\GS}=\mathcal{N}(0, \sigma^2)\boxplus \SC(s\sigma^2),
\]
here $\boxplus$ denotes the free convolution.
We give a combinatorial representation of the moments of $\mu_{\GS}$.
More specifically,
let $m_k=\int_{\mathbb{R}}x^k d\mu_{\GS}(x)$ denote the $k$-th moment of $\mu_{\GS}$,
then all odd moments vanish,
while for $k=2r$,
$$
m_{2r}
=
\sigma^{2r}
\sum_{\pi\in\mathcal P_2(2r)}
(s+1)^{\#\ic(\pi)}.
$$
This representation gives a direct combinatorial description of the Gaussian--semicircle law $\mu_{\GS}$
in terms of pairing partitions $\mathcal{P}_2(2r)$, 
complementing the usual description through non-crossing partitions.
The weight in this formula is determined by the number of isolated chords in the chord diagrams of each pairing partition $\pi$,
which is denoted by $\#\ic(\pi)$.

Our second result as presented in \cref{sec:result_EED} identifies the limiting EED of the fluctuation $W_n$.
For fixed $q\geq 0$ and $p\in(0,1)$,
let $\eta_2$ be the variance of $\Ber(p)$.
We prove that the EED of $W_n$ converges almost surely weakly to a Gaussian--semicircle law
$$
\mu_{\mathrm{Lap}}
:=
\mathcal N\bigl(0,\eta_2\bigr)
\boxplus
\mathrm{SC}\bigl((q+1)\eta_2\bigr).
$$
Thus the higher-dimensional up-Laplacian exhibits the same qualitative Gaussian--semicircle free convolution phenomenon as the Laplacian of an Erd\H{o}s--R\'enyi graph, while the variance of the semicircular component is multiplied by $q+1$.
When $q=0$,
the theorem reduces to the corresponding result of Ding and Jiang \cite{MR2759729} for Erd\H{o}s--R\'enyi graphs. For $q\geq1$,
the theorem gives an explicit limiting law for the centered fluctuation spectrum of the higher-dimensional up-Laplacian in the dense Linial--Meshulam regime.

This result can also be viewed as exhibiting a higher-dimensional analogue of the decomposition observed for graph Laplacians.
More specifically,
For an Erd\H{o}s-R\'enyi graph,
the Laplacian can be decomposed as the difference between diagonal degree matrix $D_0$ and adjacency matrix $A_0$:
\[
L_0=D_0-A_0.
\]
After centering and normalizing,
\[
\frac{L_0-\EE[L_0]}{\sqrt{n}}=\frac{D-\EE[D_0]}{\sqrt{n}} - \frac{A_0-\EE[A_0]}{\sqrt{n}}.
\]
There are therefore two fluctuation mechanisms, 
i.e.,
\[
\boxed{
\textrm{degree fluctuation} + \textrm{adjacency fluctuation}.
}
\]
The degree fluctuation $\frac{D_0-\EE[D_0]}{\sqrt{n}}$ is a diagonal matrix whose entries are sums of multiple independent Bernoulli variable divided by $\sqrt{n}$,
its EED has a Gaussian-type limit.
Meanwhile,
the adjacency fluctuation $\frac{A_0-\EE[A_0]}{\sqrt{n}}$ is a scaled random symmetric matrix, 
whose entries are independent up to symmetry.
Hence its EED has a semicircle-type limit.
The remarkable fact due to Ding and Jiang \cite{MR2759729} is that when these two fluctuations are combined,
the EED limit of $\frac{L_0-\EE[L_0]}{\sqrt{n}}$ is the Gaussian--semicircle law
\[\mathcal{N}(0, \sigma^2)\boxplus \SC(\sigma^2).
\]
The similar decomposition
\[
L_q^{\up}=D_q^{\up}-A_q^{\up} 
\]
holds in higher dimensions,
with $D_q^{\up}$ the upper-degree matrix,
and $A_q^{\up}$ the signed upper-adjacency matrix.
We assert in \cref{sec:result_EED} that even though all the entries in $W_n$ are strongly dependent on each other,
a similar Gaussian--semicircle phenomenon holds for the higher-dimensions as well.
Instead of treating these two components as independent,
we work directly with 
\[
W_n=\frac{1}{\sqrt{n}}\sum_{\tau\in \Delta_{q+1}}(\xi_\tau-p)b_\tau b_\tau^\top,
\]
here $\Delta_{q+1}$ is the set of all $(q+1)$-simplices in $\LL_n$,
 and $b_\tau$ denotes the boundary map associated with $\tau$.
Our result in \cref{sec:result_EED} reads that 
\[
\boxed{
\begin{aligned}
    \textrm{Gaussian variance}&=\eta_2;\\
    \textrm{semicircle variance}&=(q+1)\eta_2.
\end{aligned}
}
\]
Thus the Gaussian component is independent of $q$,
while the semicircular component carries the dimensional dependence through $q+1$.

The remainder of the paper is organized as follows. 
In \cref{sec:preliminary} we review the Linial--Meshulam model, 
the relevant random matrix theory notions, 
and the free probability framework. 
In \cref{sec:main_result} we state our two main results separately. 
In \cref{sec:result_moments} we give the proof of the combinatorial moment formula for the Gaussian--semicircle law.
In \cref{sec:EED_proof} we prove the limiting EED formula of $W_n$ using the moment method.
\cref{sec:moments} contains the moment calculation and the combinatorial enumeration leading to the limiting moments,
while \cref{sec:variance} proves the variance is bounded.
Based on these two results,
we give the formal proof of the limiting EED formula in \cref{sec:main_thm_proof}.

\section{Preliminary}\label{sec:preliminary}
In this section,
we review the necessary fundamental notions of the Laplacian on Linial-Meshulam model, random matrix theory, and free probability framework.
\subsection{Linial-Meshulam model\label{subsec:Y_q}}
We start with a universe containing $n$ elements.
Let $\LL_n$ be the complete $(q+1)$-simplicial complex $\LL_n$ over these $n$ elements,
that is,
$\LL_n$ is the simplicial complex containing every $(q+1)$-simplex built over those $n$ elements.
Let $C_{q}$ be the space of the $q$-chain over $\mathbb{R}$,
and let $B_{q+1}$ be the standard matrix representation of the boundary map $\partial_{q+1}:C_{q+1}\to C_q$.
It is easy to see that the matrix $B_{q+1}$ is of the shape $\binom{n}{q+1}\times \binom{n}{q+2}$.
The Linial-Meshulam model $Y_{q+1}(n, p)$ is defined by independently dropping every $(q+1)$-simplex with probability $(1-p)$.

Let $D$ be a $\binom{n}{q+2}\times \binom{n}{q+2}$ diagonal matrix,
all diagonal entries are independent and identically distributed (i.i.d.) Bernoulli random variables,
i.e.,
$D_{ii}=\xi_i\sim \Ber(p)$ for every $i\leq \binom{n}{q+2}$.
The up Laplacian $L_q^{\up}$ on the Linial-Meshulam model $Y_{q+1}(n, p)$ is defined as
\[
L_q^{\up}=B_{q+1}DB_{q+1}^\top.
\]
Since all the diagonal entries in $D$ are i.i.d. Bernoulli random variables,
we observe that the expectation $\EE[D]=pI$,
and hence $\EE [L_q^{\up}]=pB_{q+1}B_{q+1}^\top$,
whose eigenvalues contain only multiple $0$'s and $pn$'s.

\begin{lemma}\cite[Lemma 8]{MR3557457}\label{lemma:complete_up_lap}
Let $B_{q+1}$ be the $(q+1)$-boundary matrix of $\LL_n$.
We have 
\[
B_{q+1}B_{q+1}^\top=nP_q,
\]
here $n$ is the number of elements in $\LL_n$,
and $P_q$ is the orthogonal projection from $C_q(K)$ onto $\im(B_{q+1})$.
Moreover,
the eigenvalues of the projection $P_q$ are $0$ with multiplicity $\binom{n-1}{q}$ and $1$ with multiplicity $\binom{n-1}{q+1}$. 
\end{lemma}
We first centralize the up Laplacian $L_q^{\up}$.
Let \[
E_n=L_q^{\up}-\EE [L_q^{\up}]=L_q^{\up}-p\,n\,P_q,
\]
 the random matrix $E_n$ is centralized, i.e., $\EE[E_n]=0$.
We express the boundary matrix as 
$B_{q+1}=
\begin{bmatrix}
    \cdots b_\tau \cdots
\end{bmatrix}_{\tau\in \Delta_{q+1}}
$,
here each column vector $b_\tau$ is the boundary of a $(q+1)$-simplex $\tau\in \LL_n$,
and the random matrix $E_n$ can be expressed as
\begin{equation}\label{eq:E_n_decompose}
E_n=\sum_{\tau\in \Delta_{q+1}}(\xi_\tau-p)b_\tau b_\tau^\top.
\end{equation}
We observe that all the terms $(\xi_\tau-p)b_\tau b_\tau^\top$ in \cref{eq:E_n_decompose} are independent of each other since all the $\xi_\tau$ are i.i.d. according to our assumption.
We define the variance of $L_q^{\up}$ to be $\|\EE[E_n^2]\|$.

\begin{lemma}\label{lemma:E_n_var}
    We have the equality
\begin{equation}\label{eq:eq:E_n_var}
        \|\EE [E_n^2]\|=p(1-p)(q+2)n.
    \end{equation}
\end{lemma}
\begin{proof}
    According to \cref{eq:E_n_decompose},
    we have that
    \[
\|\EE \left[E_n^2\right]\|=
\left\|\sum_{\tau\in \Delta_{q+1}} \EE\left[(\xi_\tau-p)^2\right](b_\tau b_\tau^\top)^2\right\|
=
p(1-p)\left\|\sum_{\tau\in \Delta_{q+1}}(b_\tau b_\tau^\top)^2\right\|,
    \]
since $(b_\tau b_{\tau}^\top)^2=(q+2)b_\tau b_\tau^\top$, 
and according to \cref{lemma:complete_up_lap},
$\left\|\sum\limits_{\tau\in \Delta_{q+1}}b_\tau b_\tau^\top\right\|=
\|B_{q+1}B_{q+1}^\top\|=n$,
thus \cref{eq:eq:E_n_var} is proved.
\end{proof}

According to \cref{lemma:E_n_var},
the variance of $L_q^{\up}$ is $O\left(n\right)$,
thus we standardize the random matrix $E_n$ by dividing $\sqrt{n}$ and define 
\begin{equation}\label{eq:W_n}
W_n=\frac{E_n}{\sqrt{n}}=\frac{1}{\sqrt{n}}L_q^{\up}-p\sqrt{n}P_q.
\end{equation}

\subsection{Random matrix theory}\label{subsec:RMT}
The first paper on random matrices is by Hurwitz \cite{Hu97}, 
according to Diaconis and Forrester in \cite{MR3612265}.
In 1928, 
Wishart \cite{wishart} first studied random matrices for fixed size $N$,
then in 1955, 
Wigner came up with the famous semicircle law \cite{MR77805}.
Since then,
random matrix theory has become an important tool used in physics and mathematics.
We refer the reader to the literature \cite{MR2129906, MR2906465, MR4779256, MR2760897} for more details on random matrix theory.

According to $\cref{lemma:complete_up_lap}$,
we have $W_n\big|_{\ker P_q}=0$.
Hence we only need to consider the restriction $W_n\big|_{\im P_q}$.

Let 
\[
N_q=\binom{n-1}{q+1},
\]
and let 
\[
\lambda_1\leq \lambda_2\leq\cdots \leq \lambda_{N_q} 
\]
be the eigenvalues of the restriction $W_n\big|_{\im P_q}$.
Let
\[
\mu_{W_n}=\frac{1}{N_q}\sum_{i=1}^{N_q}\delta_{\lambda_i}.
\]
We refer to $\mu_{W_n}$ as the \emph{empirical eigenvalue distribution} (EED) of $W_n$.

Then the $k$-th moment of $\mu_{W_n}$ is given by
\[
\int_{\mathbb{R}}x^kd\mu_{W_n}=\frac{1}{N_q}\sum_{i=1}^{N_q}\lambda_i^k=\frac{1}{N_q}\tr(W_n^k).
\]
Observe that $\mu_{W_n}$ is a random probability measure,
and for an arbitrary function $f(x)$,
the integral
$\int_\mathbb{R}f(x)d\mu_{W_n}$ is a random variable.
We recall the definition of almost surely weak convergence in the following.

\begin{definition}\label{def:weak_converge}
We say that the random measure $\mu_{W_n}$ almost surely weakly (a.s.w.) converges to a measure $\mu$ if 
\[
\int_{\mathbb{R}}f(x)d\mu_{W_n}(x)\xrightarrow{n\to \infty} \int_{\mathbb{R}}f(x)d\mu(x) \quad \textrm{\textrm{almost surely}}
\]
for every continuous bounded test function $f\in C_b(\mathbb{R})$,
here 
\[
C_b(\mathbb{R})=\{f\in C(\mathbb{R}): \exists M > 0, \textrm{ such that } |f(x)|\leq M\quad \forall x\in \mathbb{R}\}.
\]
\end{definition}

We denote the almost surely weak convergence by 
\[
\mu_{W_n}\xrightarrow{a.s.w.}\mu.
\]
A classical example is the Gaussian Unitary Ensemble (GUE) $A_n=\frac{1}{\sqrt{n}}(a_{ij})_{i, j=1}^{n}$,
that is,
a self adjoint $n\times n$ random matrix,
such that $\{a_{ij}| i\geq j\}$ are independent standard Gaussian random variables,
which are complex for $i\neq j$ and real for $i=j$.
The Wigner's semicircle law asserts that $\mu_{A_n}\xrightarrow{a.s.w.}\mu_{\SC}$,
here $\mu_{A_n}$ is the EED associated with $A_n$:
\[
\mu_{A_n}=\frac{1}{n}\sum_{i=1}^n \delta_{\lambda_i},
\]
and
$\mu_{SC}$ is the standard semicircle law:
\[
\mu_{\SC}(x)=\frac{1}{2\pi}\sqrt{4-x^2}.
\]

\subsection{Free convolution of probability measure}\label{subsec:free_convolution}
The free probability theory was created by Dan Voiculescu in the 1980s for analyzing the von Neumann algebras of free groups.
Later on in the 1990s,
a combinatorial theory of freeness was developed by Nica and Speicher \cite{MR2266879},
who developed the lattice of non-crossing partitions and free cumulants.

In this section,
we follow mainly Novak and LaCroix's lecture notes \cite{MR3380692} to introduce mainly about free cumulants and free convolution which are necessary for describing our main result \cref{thm:main},
and refer the reader to \cite{MR2266879, MR3585560, speicher2025lecturenotesfreeprobability} for more details on free probability. 

\paragraph{Partitions and non-crossing partitions}
We first review some definitions of partitions and non-crossing partitions.
We use $[n]$ to denote the set of $\{1,2,3,\cdots, n\}$.

\begin{definition}[Partitions]\label{def:partitions}
Let $\mathcal{P}(n)$ denote the collection of all partitions of the set $[n]$,
that is,
$\pi=\{B_1, B_2,\cdots, B_k\}\in\mathcal{P}(n)$,
each $B_i\subseteq [n]$ and $B_i\neq \emptyset$,
$\cup_{i\leq k}B_i=[n]$,
and $B_i\cap B_j=\emptyset$ when $i\neq j$.
We call $B_1,\cdots, B_k$ the blocks of the partition $\pi$.
\end{definition}
Following \cref{def:partitions},
we introduce some related notation:
\begin{itemize}
    \item $\mathcal{P}_{\geq 2}(n)$ is the collection of all partitions of $[n]$  requiring that the size of each block be larger than or equal to $2$,
    i.e.,
    \[
    \mathcal{P}_{\geq 2}(n) = \{\pi\in\mathcal{P}(n): \textrm{ for each }B\in \pi, |B|\geq 2 \}.
    \]
    \item $\mathcal{P}_{2}(2n)$ is the collection of all pairing partitions of $[2n]$,
    i.e.,
    \[
    \mathcal{P}_2(2n)=\{\pi\in \mathcal{P}(2n): \textrm{for each } B\in \pi, |B|=2\}.
    \]
\end{itemize}

A pairing partition can be visualized by a chord diagram as follows.
Let $\pi\in\mathcal{P}_2(2n)$ be a partition.
We first draw the elements $1, 2, 3, \cdots, 2n$ in clockwise order along a circle,
and then connect a chord for each block in $\pi$.
We say that a chord diagram is \emph{connected} if there is no proper contiguous subinterval $\{i, i+1, \dots, i+j\} \subsetneq \{1, 2, \dots, 2n\}$ such that every chord originating inside the subinterval also ends inside it.
For example,
\cref{fig:chord} is the chord diagram associated with $\{\{1, 3\}, \{2, 8\}, \{5, 7\}, \{4, 6\}\}$, 
while it is not a connected chord diagram.
\begin{figure}[H]
    \centering

\begin{tikzpicture}[scale=1]
    \draw[draw=gray!50, thin] (0,0) circle (2cm);

    \def\n{8}

    \foreach \i in {1,...,\n} {
        \pgfmathsetmacro{\angle}{90 - (\i - 1) * 360 / \n}
        \node[circle, fill=black, inner sep=2pt, label=\angle:\textbf{\i}] (N\i) at (\angle:2cm) {};
    }


    \draw[thick] (N1.center) -- (N3.center);

    \draw[thick] (N2.center) -- (N8.center);

    \draw[thick] (N4.center) -- (N6.center);

    \draw[thick] (N5.center) -- (N7.center);

\end{tikzpicture}
\caption{The chord diagram of $\{\{1, 3\}, \{2, 8\}, \{5, 7\}, \{4, 6\}\}\in\mathcal{P}_{2}(8)$.}
\label{fig:chord}
\end{figure}
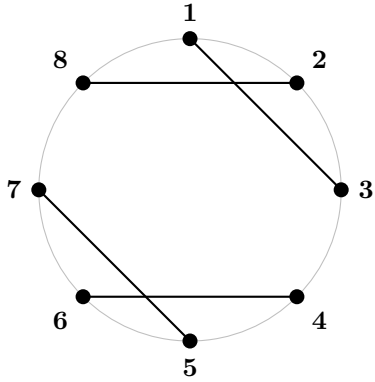
Let $\pi\in \mathcal{P}_2(2r)$ be a pairing partition of $2r$ elements $\{1, 2,\cdots, 2r\}$.
 We say that two pairing blocks $\{a_1, b_1\}, \{a_2, b_2\}\in \pi$ are \emph{crossing} if 
    \[
    a_1< a_2 < b_1 < b_2.
    \]

    \begin{definition}[isolated chord]\label{def:iso_chord}
    We say a chord $\{a, b\}$ is \emph{isolated} if it is not crossing with any other pairing in $\pi$,
    and we denote by $\ic(\pi)$ the set of all isolated chords in $\pi$:
    i.e.,
    \[
    \ic(\pi)=
    \left\{\{a,b\}\in \pi: \{a, b\} \textrm{ crosses no other chord of  } \pi   
    \right\}.
    \]
    \end{definition}

\begin{definition}[Non-crossing partitions]\label{def:non_crossing}
    Let $\pi\in\mathcal{P}(n)$.
If we can find out $a<b<c<d$ such that $a$ and $c$ are in one block $B_1\in \pi$,
$b$ and $d$ are in another block $B_2\in\pi$,
we say $B_1$ and $B_2$ cross.
If there is no pair of blocks crossing,
we say $\pi$ is non-crossing.
We denote the set of all non-crossing partitions of $[n]$ by $\NC(n)$.
\end{definition}
We denote by $\NCev(2n)$ the collection of non-crossing partitions of $[2n]$ that requires the size of each block to be even,
i.e.,
\[
\NCev(2n)=\{\pi\in \NC(2n): \textrm{ for each } B\in \pi, |B| \textrm{ is even}\}.
\]

A geometric intuition of a non-crossing partition is as follows.
Let $\pi\in\NC(n)$ be a partition.
We again draw the elements $1, 2, 3, \cdots, n$ in clockwise order along a circle.
Then for each block $B\in\pi$,
we draw the convex hull connecting all elements in $B$.
The partition $\pi$ is non-crossing if and only if none of the convex hulls intersect each other inside the circle.
See \cref{fig:non_crossing} for an example of non-crossing partition visualization.
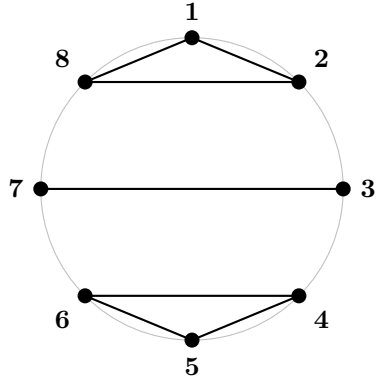
\begin{figure}[H]
    \centering
\begin{tikzpicture}[scale=1]
    \draw[draw=gray!50, thin, fill=white] (0,0) circle (2cm);

    \def\n{8}

    \foreach \i in {1,...,\n} {
        \pgfmathsetmacro{\angle}{90 - (\i - 1) * 360 / \n}
        \node[circle, fill=black, inner sep=2pt, label=\angle:\textbf{\i}] (N\i) at (\angle:2cm) {};
    }


    \draw[thick] (N1.center) -- (N2.center) -- (N8.center) -- cycle;

    \draw[thick] (N3.center) -- (N7.center);

    \draw[thick] (N4.center) -- (N5.center) -- (N6.center) -- cycle;
\end{tikzpicture}
\caption{The visualization of $\{\{1, 2, 8\}, \{3, 7\}, \{4, 5, 6\}\}\in\NC(8)$.}\label{fig:non_crossing}
\end{figure}

\paragraph{Cumulants}
We first review the cumulants in classical probability theory.
Let $m_n=\EE[X^n]$ be the $n$-th moment of a random variable $X$.
The cumulants $c_n(X)$ are defined by the recurrences
\begin{equation}\label{eq:cumulant}
m_n(X)=\sum_{\pi\in \mathcal{P}(n)}\prod_{B\in \pi} c_{|B|}(X).
\end{equation}
The first two cumulants are in fact the mean and variance of $X$:
\begin{equation*}
			\begin{aligned}
			c_1(X) &=m_1(X); \\
			c_2(X) &=m_2(X)-m_1(X)^2;\\
			\end{aligned}
		\end{equation*}
One reason why the cumulants are of interest is that the cumulants are simpler than moments.
Consider the standard Gaussian $X\sim\mathcal{N}(0, 1)$ as an example,
the moments of $X$ are given by
\[
m_n(X)=
\begin{cases}
    0 & \textrm{if}\quad n=2k+1\\
    (2k-1)!! & \textrm{if}\quad n=2k.
\end{cases}
\]
In contrast to that,
the cumulants of $X$ are given by
\[
c_n(X)=
\begin{cases}
    1 & \textrm{if}\quad n=2\\
    0 & \textrm{else}.
\end{cases}
\]
\paragraph{Independence}
We can extend the recurrence relationship \cref{eq:cumulant} to that between mixed moments and mixed cumulants.
Let $X_1,\cdots, X_n$ be $n$ random variables (not necessarily distinct).
The mixed moment is defined as
\[
m_n(X_1, \cdots, X_n)=\EE[X_1 \cdots X_n],
\]
and the mixed cumulants are defined recursively by
\[
m_n(X_1\cdots, X_n)=\sum_{\pi\in \mathcal{P}(n)}\prod_{B\in \pi}c_{|B|}(X_i: i\in B).
\]
Following this definition,
the second mixed cumulant of $X_1$, $X_2$ is their covariance:
\[
c_2(X_1, X_2)=m_2(X_1, X_2)-m_1(X_1)m_1(X_2).
\]

There is a fundamental relationship between mixed cumulants and independence.
Two random variables $X_1$ and $X_2$ are independent if and only if all their mixed cumulants vanish,
\begin{equation*}
				\begin{aligned}
					&c_2(X_1,X_2) = 0; \\
					&c_3(X_1,X_1,X_2) = c_3(X_1,X_2,X_2) = 0; \\
					&\vdots
				\end{aligned}
			\end{equation*}	
Hence the cumulants are additive:
\[
c_n(X_1+X_2) = c_n(X_1) + c_n(X_2) \textrm{ for each } n.
\]
\paragraph{Free cumulants}
The free cumulants are also given by recurrence relations similar to \cref{eq:cumulant},
the only difference is that the sum is restricted to non-crossing partitions instead of all partitions.
More rigorously,

We can now give the formal definition of free cumulants.
\begin{definition}\label{def:free_cumulant}
Let $X$ be a random variable with moments $m_n(X)$.
We define the free cumulants $\kappa_n(X)$ of $X$ recursively:
\[
m_n(X)=
\sum_{\pi\in \NC(n)}\prod_{B\in \pi} \kappa_{|B|}(X).
\]
\end{definition}

A free probability analogue of the standard Gaussian is a random variable whose free cumulants are 
\[
0, 1, 0, 0, \cdots,
\]
it turns out that this is in fact the Wigner's semicircle distribution.
In fact,
let $\SC(\sigma^2)$ be the semicircle distribution with variance $\sigma^2$,
the probability density function is 
\[
\mu_{\SC}(x)=\frac{1}{2\pi\sigma^2}\sqrt{4\sigma^2-x^2}.
\]

\begin{lemma}\label{lemma:semicircle_cumulants}
    Let $X\sim\SC(\sigma^2)$.
    Then the free cumulants $\{\kappa_n(X)\}_{n\geq 1}$ are given by 
    \[
    \kappa_n(X)=
    \begin{cases}
        \sigma^2 & \textrm{ if}\quad n=2;\\
        0 & \textrm{else}.
    \end{cases}
    \]
\end{lemma}

In contrast to that,
the free cumulants of Gaussian variable are not as simple as  the classical cumulant,
they are in fact given by the number of connected pairing partitions.
We denote by $\mathcal{P}_{2}(2n)$ the collection of all pairing partitions of $2n$ elements, 
that is,
\[
\mathcal{P}_{2}(2n)=\{\pi\in \mathcal{P}(2n): \textrm{each block } B\in\pi, |B|=2\}.
\]

According to \cite{MR1938356},
we can describe the free cumulants of a Gaussian variable $X\sim \mathcal{N}(0, \sigma^2)$ as follows:
\begin{lemma}\label{lemma:gauss_free_cumulant}
    Let $X\sim\mathcal{N}(0, \sigma^2)$ be a Gaussian distribution.
    Then the free cumulants $\{\kappa_n(X)\}_{n\geq 1}$ are given by
    \[
    \kappa_n(X)=
    \begin{cases}
        C_r \sigma^{2r} & n=2r \textrm{ is even};\\
        0 & n \textrm{ is odd},
    \end{cases}
    \]
    here $\{C_r\}_{r\geq 1}$ is the number of connected chord diagrams with $2r$ elements. 
\end{lemma}

\begin{remark}\label{rmk:A000699}
The sequence $\{C_r\}_{r\geq 1}$ is referred to as $\texttt{A000699}$ in Sloane's Online Encyclopedia of Integer Sequences \cite{oeisA000699};
the first few terms are
\[
1, 1, 4, 27, 248, 2830, \cdots.
\]
\end{remark}

\paragraph{Free convolution}
If $X$ and $Y$ are two classical independent random variables with distributions $\mu_X$ and $\mu_Y$ respectively,
 the distribution of $X+Y$ is the convolution $\mu_X*\mu_Y$.
In the framework of free probability theory,
one can also define the \emph{free independence} of noncommutative random variables.
In analogy to the classical probability theory,
if $X$ and $Y$ are \emph{freely independent},
the free cumulants are also additive:
\[
\kappa_n(X+Y)=\kappa_n(X)+\kappa_n(Y),
\]
in this case,
the distribution of the summed noncommutative random variable $X+Y$ is referred to as the \emph{free convolution} of $\mu_X$ and $\mu_Y$,
and it is denoted as $\mu_X\boxplus \mu_Y$.

We will not present the formal definition of free independence and free convolution.
Instead, we give an intuitive example to explain it.
Let $A_n=\frac{1}{\sqrt{n}}(a_{ij})_{i,j=1}^n$ be a standard GUE matrix,
$B_n=(b_{ij})_{i,j=1}^n$ be  a diagonal matrix,
and all the diagonal entries $\{b_{ii}\}_{i= 1}^n$ are i.i.d real standard Gaussian.
If the entries of $A_n$ and $B_n$ are independent,
then the EED of $A_n+B_n$ converges a.s.w. to $\mathcal{N}(0, 1)\boxplus \SC(1)$.

\section{Main Results}\label{sec:main_result}

We now state the main results of the paper. 
Our first theorem gives an explicit combinatorial representation of the moments of the Gaussian--semicircle law. 
Our second theorem identifies that the limiting EED of $W_n$ is equal to $\mu_{\lap}$ in the sense of almost surely weak convergence.
These results provide a direct connection between the geometry of simplicial complexes, pairing partitions, and free probability.

\subsection{The Gaussian--semicircle law}\label{sec:result_moments}
In this section we will show that the Gaussian--semicircle law admits a purely combinatorial description. 
Recall \cref{def:iso_chord} that $\ic(\pi)$ is the set of isolated chords in a pairing partition $\pi$.
\begin{theorem}[Combinatorial moment formula]\label{thm:moments}
Let $\sigma>0$, $s>0$.
Let
\[
\mu_{\GS}
=
\mathcal N\bigl(0,\sigma^2\bigr)
\boxplus
\operatorname{SC}\bigl(s\sigma^2\bigr)
\] be a Gaussian--semicircle law associated with $\sigma$ and $s$,
and $m_{k}=\int_{\mathbb{R}}x^{k}d\mu_{\GS}(x)$ be the $k$-th moment of $\mu_{\GS}$.
Then we have
\begin{equation}\label{eq:moments}
m_{k}=
\begin{cases}
    0 & \textrm{if}\quad k= 2r+1;\\
    \sigma^{2r} \sum\limits_{\pi\in \mathcal{P}_2(2r)}(s+1)^{\# \ic(\pi)} & \textrm{if} \quad k=2r.
\end{cases}
\end{equation}
\end{theorem}

\begin{proof}
Let $\{\kappa_n\}_{n\geq 1}$ be the free cumulants of $\mu=\mathcal{N}(0, \sigma^2)\boxplus \SC(s\sigma^2)$.
According to \cref{lemma:semicircle_cumulants} and \cref{lemma:gauss_free_cumulant},
we have 
\[
\kappa_k=
\begin{cases}
   (s+1)\sigma^2 & k = 2;\\
   C_r\sigma^{2r} & k=2r > 2;\\
   0 & k $\textrm{ is odd}$.
\end{cases}
\]
    Recall the recursion formula of free cumulants:
    \[
    m_{k}=\sum_{\pi\in \NC(k)}\prod_{B\in \pi} \kappa_{_{|B|}},
    \]
    it is obvious that all the odd moments of $\mu_{\lap}$ vanish,
    and when $k=2r$ is even,
    we can restrict the sum to the even non-crossing partitions and obtain:
    \[
\begin{aligned}
m_{2r}
&=\sum_{\sigma\in \NCev(2r)}\prod_{B\in \sigma}\kappa_{|B|}\\
&=\sum_{\sigma\in\NCev(2r)}\prod_{\substack{B\in\sigma\\|B|=2}}\kappa_{2}\prod_{\substack{B\in\sigma\\|B|\geq 4}}\kappa_{|B|}.
\end{aligned}
\]
hence 
\[
\begin{aligned}
m_{2r}
&=\sum_{\sigma\in\NCev(2r)}\prod_{\substack{B\in\sigma\\|B|=2}}(s+1)\sigma^2\prod_{\substack{B\in\sigma\\|B|\geq 4}}C_{|B|/2}\sigma^{|B|}\\
&=
\sigma^{2r}\sum_{\sigma\in\NCev(2r)}\prod_{\substack{B\in\sigma\\|B|=2}}(s+1)\prod_{\substack{B\in\sigma\\|B|\geq 4}}C_{|B|/2}.
\end{aligned}
\]

We define a map
\[
\begin{aligned}
    \Phi: \mathcal{P}_2(2r)&\to \NCev(2r)\\
    \pi & \mapsto \Phi(\pi)
\end{aligned}
\]
by merging all the connected chords into one single block.
since the size of each chord is $2$, the map $\Phi$ is surjective.
It is easy to see that 
\[\#\ic(\pi_1)=\#\ic(\pi_2) \textrm{ if }
 \Phi(\pi_1)=\Phi(\pi_2),
\]
hence 
\[
\prod_{\substack{B\in\sigma\\|B|=2}}(s+1)=(s+1)^{\ic\left(\Phi^{-1}(\sigma)\right)},
\]
recall that $C_1=1$,
hence
\[
\prod_{\substack{B\in\sigma\\|B|\geq 4}}C_{|B|/2}
=
\prod_{B\in\sigma}C_{|B|/2}=
\sum_{\pi\in\Phi^{-1}(\sigma)}1.
\]
Thus
\[
\begin{aligned}
m_{2r}
&=
\sigma^{2r}\sum_{\sigma\in \NCev(2r)}(s+1)^{\#\ic(\Phi^{-1}(\sigma))}\sum_{\pi\in \Phi^{-1}(\sigma)}1\\
&=
\sigma^{2r}\sum_{\substack{\pi\in \Phi^{-1}(\sigma)\\\sigma\in \NCev(2r)}}(s+1)^{\#\ic(\Phi^{-1}(\sigma))}\\
&=
\sigma^{2r}\sum_{\pi\in\mathcal{P}_2(2r)}(s+1)^{\#\ic(\pi)}.
\end{aligned}
\]
\end{proof}

Notice that in \cref{eq:moments} the sum is taken over all pairing partitions, 
rather than only over non-crossing pairings as commonly used in the area of free probability,
and the statistic that determines the contribution of a pairing is the number of isolated chords.

\begin{example}
We compute the $4$-th moment $m_4$ as an example to explain \cref{eq:moments} more explicitly.
When $k=2r=4$,
there are $(2\times 2-1)!!=3$ different pairing partitions,
as displayed in \cref{fig:P2_4}.
We observe that the partitions $\{\{1,2\}, \{3, 4\}\}$ and $\{\{1,4\}, \{2, 3\}\}$ have $2$ isolated chords for each,
and the partition $\{\{1,3\}, \{2, 4\}\}$ has no isolated chord.
Hence according to \cref{thm:moments} the $4$-th moment is
\[
m_{4}=\sigma^4\left(2(s+1)^2+1\right).
\]
\tikzset{
    partition_circle/.style={draw=gray!50, thin, fill=white},
    vertex/.style={circle, draw=black, fill=black, inner sep=1.2pt},
    label_node/.style={font=\small},
    chord/.style={thick, black!90}
}

\newcommand{\drawPairing}[1]{%
    \begin{tikzpicture}[baseline=(current bounding box.center)]
        \def\r{1.3} 
        \draw[partition_circle] (0,0) circle (\r);
        
        \foreach \i in {1,...,4} {
            \pgfmathsetmacro{\angle}{135 - (\i-1)*90}
            \node[vertex] (v\i) at (\angle:\r) {};
            \node[label_node] at (\angle:\r+0.3) {\i};
        }
        
        \foreach \u/\v in {#1} {
            \draw[chord] (v\u) -- (v\v);
        }
    \end{tikzpicture}
}

\begin{figure}[H]
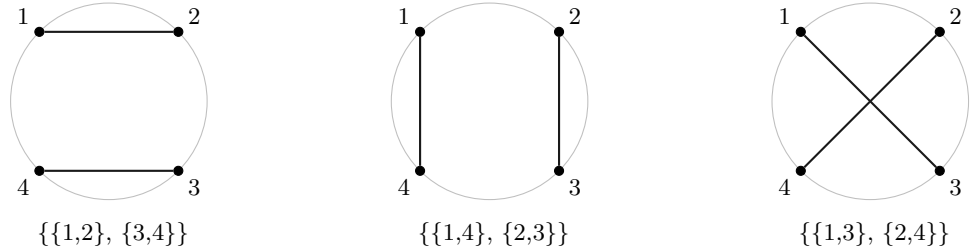

    \centering
    \begin{subfigure}[b]{0.3\textwidth}
        \centering
        \drawPairing{1/2, 3/4}
        \caption*{\{\{1,2\}, \{3,4\}\}}
    \end{subfigure}
    \hfill
    \begin{subfigure}[b]{0.3\textwidth}
        \centering
        \drawPairing{1/4, 2/3}
        \caption*{\{\{1,4\}, \{2,3\}\}}
    \end{subfigure}
    \hfill
\begin{subfigure}[b]{0.3\textwidth}
        \centering
        \drawPairing{1/3, 2/4}
        \caption*{\{\{1,3\}, \{2,4\}\}}
    \end{subfigure}
    \caption{The pairing partitions $\mathcal{P}_2(4)$.}
    \label{fig:P2_4}
\end{figure}
\end{example}

\subsection{Limiting eigenvalue distribution}\label{sec:result_EED}

The main result of the limiting EED of $W_n$ is the following.

\begin{theorem}[Limiting spectral distribution]\label{thm:main}
Let \(q\geq0\) and \(p\in(0,1)\) be fixed. Let \(W_n\) be the standardized up Laplacian of the Linial--Meshulam random complex \(Y_{q+1}(n,p)\), 
let $\eta_2=p(1-p)$ be the variance of $\Ber(p)$,
and let \(\mu_{W_n}\) be its empirical eigenvalue distribution on \(\operatorname{im}(P_q)\). 
Then
\begin{equation}\label{eq:free_conv}
\mu_{W_n}
\xrightarrow{\mathrm{a.s.w.}}
\mu_{\mathrm{Lap}},
\end{equation}
where
\[
\boxed{
\mu_{\mathrm{Lap}}
=
\mathcal N\bigl(0,\eta_2\bigr)
\boxplus
\operatorname{SC}\bigl((q+1)\eta_2\bigr).
}
\]
\end{theorem}

The limiting measure $\mu_{\lap}$ has two distinct components. 
When \(q=0\), the up Laplacian is the ordinary graph Laplacian, and \cref{thm:main} reduces to the Gaussian--semicircle limit obtained by Ding and Jiang \cite{MR2759729} for the Erd\H{o}s--R\'enyi model. For \(q\geq1\), the theorem gives the corresponding fluctuation law for the higher-dimensional up Laplacian. 
The Gaussian component has variance $\eta_2$,
while the semicircular component has variance $(q+1)\eta_2$.
Thus,
the higher-dimensional boundary structure contributes an additional factor \(q+1\) to the semicircular part of the limiting spectrum.


The proof of \cref{thm:main} is completed in \cref{sec:main_thm_proof} using the moment methods. 
There we establish the convergence of the expected moments in \cref{sec:moments}.
Together with the variance bound for the scaled moments in \cref{sec:variance}. 
The moment-determinacy of the limiting measure thus yields the almost surely weak convergence of every trace moment stated in \cref{thm:main}.

The appearance of the factor \(q+2\) in the moments of $\mu_{\lap}$ is a geometric feature of the simplicial complex. Every \((q+1)\)-simplex contains exactly \(q+2\) different \(q\)-faces.
Whenever an isolated chord produces a repeated simplex in the trace expansion, the corresponding contribution acquires a factor \(q+2\). The higher-dimensional geometry is therefore recorded in the pairing formula through the weight
$$
(q+2)^{\#\ic (\pi)}.
$$

\cref{thm:moments} gives a combinatorial representation of the moments of the candidate limiting measure, while \cref{thm:main} identifies this measure as the almost-sure limiting EED of $W_n$.
We give an intuitive visualization of this limiting EED below. 
\begin{example}
We set $p=0.3$, $q=1$.
We draw the EED of $W_n$,
for different $n=25$, $50$, $100$ and $200$ respectively (see \cref{fig:histgrams}). 
The ideal probability measure is 
\[
\mu_{\lap}=\mu_{\operatorname{Gauss}}\boxplus \mu_{\SC}
=\mathcal{N}(0, 0.21)\boxplus\SC(0.42).
\]
As a comparison,
we sum up a $10000\times 10000$ GUE matrix with variance $0.42$,
and a Gaussian diagonal matrix with variance $0.21$.
We draw the EED of this summed matrix in gray.
We notice that when $n$ increases,
the EED of $W_n$ converges to $\mu_{\lap}$.
\begin{figure}[H]
  \centering
  \subcaptionbox{Model $Y_2(25, 0.3)$}
  {%
    \includegraphics[width=0.4\textwidth]{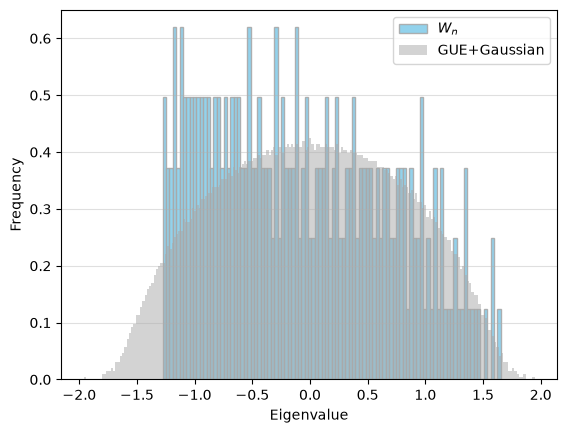}%
  }\hfill
  \subcaptionbox{Model $Y_2(50, 0.3)$}
  {%
    \includegraphics[width=0.4\textwidth]{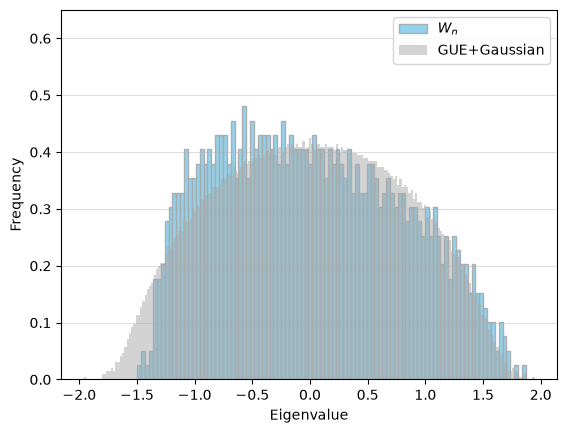}%
  }
  
  \vspace{1em} 
  
  \subcaptionbox{Model $Y_2(100, 0.3)$}
  {%
    \includegraphics[width=0.4\textwidth]{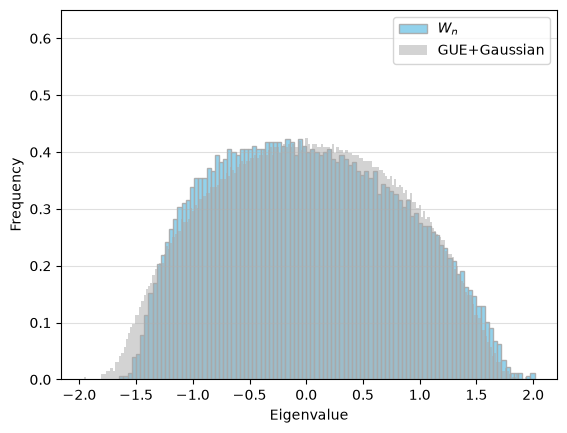}%
  }\hfill
  \subcaptionbox{Model $Y_2(200, 0.3)$}
  {%
    \includegraphics[width=0.4\textwidth]{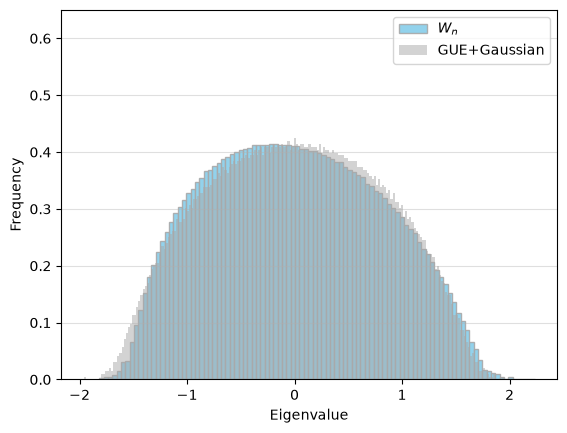}%
  }
  
  \caption{The EED of $\mu_{W_n}$, for different $n=25, 50, 100, 200$, respectively. }
  \label{fig:histgrams}
\end{figure}
\end{example}


\section{Proof of the limiting eigenvalue distribution}\label{sec:EED_proof}

We use the canonical \emph{moment method} to prove \cref{thm:main},
the moment method is widely used in random matrix theory,
we refer the reader to read \cite[Exercise 2.1.16]{MR2760897} for the case of the classical Wigner matrix and \cite[Theorem 3.1]{MR3689342} for the case of the adjacency matrix of $Y_{q+1}(n, p)$.
We first prove the limiting moments converge to the moments of $\mu_{\lap}$ in \cref{sec:moments},
and then we prove that the variance of the moments are bounded in \cref{sec:variance}.
Finally,
we present the proof of \cref{thm:main} in \cref{sec:main_thm_proof}.

\subsection{Limit of the moments}\label{sec:moments}

\begin{proposition}[Moments]\label{prop:moments}
Let 
\[
m_k^{(n)}=\frac{1}{N_q}\EE[\tr(W_n^{k})]
\]
be the scaled expectation of $k$-th moment of $W_n$.
Let $m_k$ be the $k$-th moment of $\mu_{\lap}$.
Then 
\[
\lim_{n\to\infty}m_k^{(n)}=m_k.
\]
\end{proposition}

We first prove that the moment $m_k^{(n)}=\frac{1}{N_q} 
\mathbb E\!\left[\operatorname{tr}(W_n^k)\right]$ converges to $0$ if $k$ is odd.
\begin{lemma}\label{lemma:odd_moments}
If $k$ is an odd natural number, 
\[
\lim_{n\to \infty}m_k^{(n)}=0.
\]
\end{lemma}

\begin{proof}
Recall that 
\[
W_n=\frac{1}{\sqrt{n}}\sum_{\tau\in \Delta_{q+1}}(\xi_\tau-p)b_\tau b_\tau^\top,
\]
since $\xi_\tau\sim\Ber(p)$ are i.i.d.,
the expectation $\EE[\xi_\tau-p]=0$ for all $\tau\in \Delta_{q+1}$,
any partition with a singleton block vanishes,
thus
we can express  
\begin{equation}\label{eq:moments_pair}
\begin{aligned}
m_k^{(n)}=
\EE\left[\frac{1}{N_q}\tr(W_n^k)\right]
&=
\frac{1}{N_q n^{k/2}}
\sum_{\tau_1,\cdots, \tau_k\in \Delta_{q+1}}\EE\left[\prod_{i=1}^k(\xi_{\tau_i}-p)\right]\tr\left(\prod_{i=1}^kb_{\tau_{i}}b_{\tau_{i}}^\top\right)\\
&=
\frac{1}{N_q n^{k/2}}\sum_{\pi\in 
\mathcal{P}_{\geq 2}(k)}
\prod_{S\in \pi}\eta_{|S|}\sum_{\substack{\tau_1\cdots \tau_{|\pi|}\\\textrm{ distinct, connected}\\\textrm{via }q\textrm{-faces}}}\prod_{i=1}^k b_{\tau_{\pi(i)}}^\top b_{\tau_{\pi(i+1)}},
\end{aligned}
\end{equation}
here the indices $i=1,\cdots, k$ read cyclically,
i.e.,
$b_{\tau_{\pi(k+1)}}=b_{\tau_{\pi(1)}}$,
and
\begin{itemize}
    \item $\eta_m$ is the $m$-th centralized moment of $\Ber (p)$;
    \item $\pi(i)$ is the block of partition $\pi$ that contains $i$.
\end{itemize}
We need to evaluate the leading order of $\EE[\tr(W_n^k)]$.
Since $\pi\in \mathcal{P}_{\geq 2}(k)$,
we have $|\pi|\leq \frac{k}{2}$.
We observe that $\prod_{S\in \pi}\eta_{|S|}$ is the product of several moments of $\Ber(p)$,
and $\pi\in \mathcal{P}_{\geq 2}(k)$ is independent of $n$,
hence $\prod_{S\in \pi}\eta_{|S|}=O(1)$.
Besides that,
each term $|b_{\tau_{\pi(i)}}^\top b_{\tau_{\pi(i+1)}}|\leq q+2$, 
thus $\prod_{i=1}^k b_{\tau_{\pi(i)}}^\top b_{\tau_{\pi(i+1)}}=O(1)$.

Other than that,
if $\prod_{i=1}^k b_{\tau_{\pi(i)}}b_{\tau_{\pi(i+1)}}^\top\neq 0$,
then consecutive $(q+1)$-simplices in the sequence $\{\tau_{\pi(1)}, \tau_{\pi(2)},\cdots, \tau_{\pi(k-1)}, \tau_{\pi(k)}\}$ must be equal or upper adjacent,
hence the distinct simplices form a connected walk in the Johnson graph $J(n, q+2, q+1)$.
We first randomly choose one $(q+1)$-simplices,
there are $\binom{n}{q+2}$ options,
we then need to choose $|\pi|-1$ different $(q+1)$-simplices which are consecutively upper adjacent,
in each step, we choose one $q$-face from the previous $(q+1)$-simplices, since there are at most $(q+2)$ different $q$-faces,
the option is $O(1)$,
we then only need to choose one additional element from the remaining elements which are not chosen yet, the order is $O(n)$,
therefore,
the total order of the remaining $|\pi|-1$ steps is $O(n^{|\pi|-1})$ and 
\[
\sum_{\substack{\tau_1\cdots \tau_{|\pi|}\\\textrm{ distinct, connected}\\\textrm{via }q\textrm{-faces}}}\prod_{i=1}^k b_{\tau_{\pi(i)}}^\top b_{\tau_{\pi(i+1)}}=\binom{n}{q+2}\cdot O(n^{|\pi|-1})=O(n^{q+2+|\pi|-1}),
\]
according to \cref{eq:moments_pair} we obtain
\begin{equation}\label{eq:moment_order}
m_k^{(n)}=O(n^{|\pi|-k/2}).
\end{equation}
When $k$ is odd,
$|\pi|<k/2$ and thus $m_k^{(n)}\to 0$ as $n\to \infty$.
\end{proof}
We thus only need to analyze the even moments.
Let $k=2r$ be an even number.
\begin{lemma}\label{lemma:even_moments}
We have 
\begin{equation}\label{eq:tilde_moment_orig}
m_{2r}^{(n)}\sim
\frac{(\eta_{2})^r }{N_qn^r}\sum_{\pi\in \mathcal{P}_2(2r)}\sum_{\substack{\tau_1\cdots \tau_{r}\\\textrm{ distinct, connected}\\\textrm{via }q\textrm{-faces};\\|\tau_1\cup\cdots \cup\tau_r|=q+r+1}}\prod_{i=1}^{2r} b_{\tau_{\pi(i)}}^\top b_{\tau_{\pi(i+1)}},\,\, \textrm{ as } n\to \infty.
\end{equation}
\end{lemma}

\begin{proof}
 According to \cref{eq:moment_order} in the proof of \cref{lemma:odd_moments},
    \[
    |\pi|\leq r \textrm{ and }
    m_k^{(n)}=O(n^{|\pi|-r}).
    \]
 Hence, a partition $\pi$ can contribute to the leading term of $\EE[\tr(W_n^{2r})]$ if and only if $|\pi|=r$.
    On the other hand, since the partition $\pi$ has no singleton block,
    we obtain that $\pi$ must be a pairing partition,
therefore
\[
m_{2r}^{(n)}
\sim
\frac{(\eta_{2})^r }{N_qn^r}\sum_{\pi\in \mathcal{P}_2(2r)}\sum_{\substack{\tau_1\cdots \tau_{r}\\\textrm{ distinct, connected}\\\textrm{via }q\textrm{-faces}}}\prod_{i=1}^{2r} b_{\tau_{\pi(i)}}^\top b_{\tau_{\pi(i+1)}} \,\,\textrm{ as } n\to \infty.
\]

In addition to that,
we observe that if $\tau_1\cdots, \tau_r$ are $r$ distinct $(q+1)$-simplices which are consecutively upper adjacent,
then the support size $|\tau_1\cup\cdots\cup\tau_r|\leq q+r+1$,
and
\begin{equation}\label{eq:sum_order}
\sum_{\substack{\tau_1\cdots \tau_{r}\\\textrm{ distinct, connected}\\\textrm{via }q\textrm{-faces}}}\prod_{i=1}^{2r} b_{\tau_{\pi(i)}}^\top b_{\tau_{\pi(i+1)}}=O(n^{|\tau_1\cup\cdots\cup\tau_r|})\leq O(n^{q+r+1}).
\end{equation}
Thus only when $|\tau_1\cup\cdots\cup\tau_r|= q+r+1$ can \cref{eq:sum_order} contribute to the leading term of $m_{2r}^{(n)}$.
\end{proof}

We cyclically order $2r$ elements 
$1\rightarrow2\rightarrow\cdots\rightarrow2r\rightarrow 1$.
Let $\pi\in \mathcal{P}_2(2r)$.
We define a multi-edge graph whose vertices are all the pairing blocks,
and draw an edge $\pi(i)\sim\pi(i+1)$ for each $i=1,\cdots 2r$.
We denote the resulting graph as $H_\pi$.
The graph $H_{\pi}$ is a multi-edge graph containing $2r$ edges in total,
the degree of each vertex is $4$, 
and self-loop is allowed.
We also denote by $G_\pi$ the simple graph associated with $H_\pi$,
i.e.,
$G_\pi$ is obtained from $H_\pi$ by removing all self-loops and merging  multi-edges into one edge.

\begin{example}\label{ex:graph_H_pi}
    If $r=3$,
    and $\pi=\{\{1, 3\}, \{2, 6\}, \{4, 5\}\}$,
    the graphs $H_\pi$ and $G_\pi$ are as follows:
\tikzset{
    partition_circle/.style={draw=gray!50, thin, fill=white},
    vertex/.style={circle, draw=black, fill=black, inner sep=1.2pt},
    label_node/.style={font=\footnotesize},
    chord/.style={thick, black!90}
}

\newcommand{\drawPairing}[2]{%
    \begin{tikzpicture}[baseline=(current bounding box.center)]
        \def\r{1.0} 
        \draw[partition_circle] (0,0) circle (\r);
        
        \foreach \i in {1,...,6} {
            \pgfmathsetmacro{\angle}{150 - (\i-1)*60}
            \node[vertex] (v\i) at (\angle:\r) {};
            \node[label_node] at (\angle:\r+0.25) {\i};
        }
        
        \foreach \u/\v in {#2} {
            \draw[chord] (v\u) -- (v\v);
        }
    \end{tikzpicture}
}

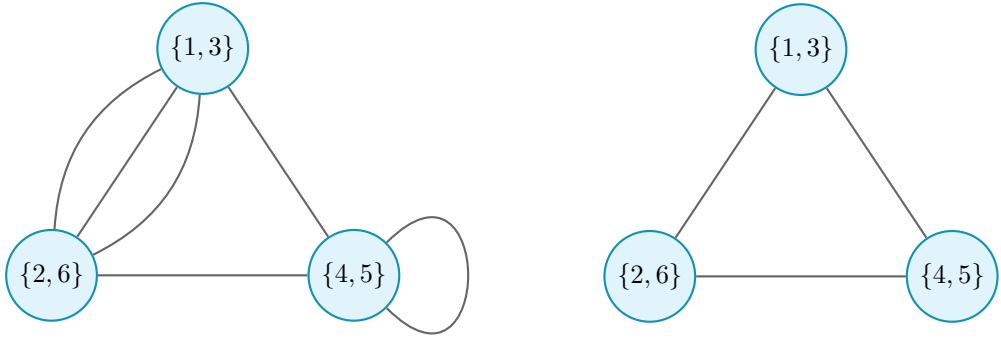
\begin{figure}[H]
    \centering
    
    \tikzset{
        vertex/.style={
            circle, 
            draw=cyan!70!black, 
            fill=cyan!10, 
            thick, 
            minimum size=1.2cm, 
            inner sep=0pt
        },
        edge/.style={
            thick, 
            draw=gray!80!black
        }
    }

    \begin{subfigure}[b]{0.45\textwidth}
        \centering
        \begin{tikzpicture}
            \node[vertex] (A) at (0, 3) {$\{1, 3\}$};
            \node[vertex] (B) at (-2, 0) {$\{2, 6\}$};
            \node[vertex] (C) at (2, 0) {$\{4, 5\}$};

            \draw[edge] (A) to[bend right=30] (B);
            \draw[edge] (A) to (B);
            \draw[edge] (A) to[bend left=30] (B);

            \draw[edge] (B) -- (C);
            \draw[edge] (C) -- (A);

            \path[edge] (C) edge [out=-45, in=45, loop, looseness=6] (C);
        \end{tikzpicture}
        \caption{The multi-edge graph $H_\pi$.}
    \end{subfigure}
    \hfill 
    \begin{subfigure}[b]{0.45\textwidth}
        \centering
        \begin{tikzpicture}
            \node[vertex] (A) at (0, 3) {$\{1, 3\}$};
            \node[vertex] (B) at (-2, 0) {$\{2, 6\}$};
            \node[vertex] (C) at (2, 0) {$\{4, 5\}$};

            \draw[edge] (A) -- (B);
            \draw[edge] (B) -- (C);
            \draw[edge] (C) -- (A);
            
            \path (C) edge [draw=none, out=-45, in=45, loop, looseness=6] (C);
            
        \end{tikzpicture}
        \caption{The simple graph $G_\pi$.}
    \end{subfigure}
    
    \caption{Graphs $H_\pi$ and $G_\pi$ in \cref{ex:graph_H_pi}.}
    \label{fig:graph-comparison}
\end{figure}
\end{example}

It is obvious to see that the value of the product $\prod_{i=1}^{2r}b_{\tau_{\pi(i)}}^\top b_{\tau_{\pi(i+1)}}$ relies on both the structure of graph $H_\pi$ and the $(q+1)$-simplices $\tau_1,\cdots, \tau_{2r}$.
We need to use the Johnson graph to analyze the structure of the underlying simple graph $G_\pi$.
Before that,
we shall first recall the definition of Johnson Graphs (see \cite{MR1829620}) below.
\begin{definition}
The Johnson graph $J(n, k, l)$ is the graph whose vertices are the $k$-element subsets of an $n$-element set,
where two vertices are adjacent if their intersection has size $l$.
\end{definition}

We can identify each $(q+1)$-simplex in $\LL_n$ as a vertex in $J(n, q+2, q+1)$,
and two $(q+1)$-simplices in $\LL_n$ are upper adjacent  if and only if their corresponding vertices in $J(n, q+2, q+1)$ are adjacent.

Let $\pi\in \mathcal{P}_2(2r)$,
    and $H_\pi$, $G_\pi$ be the multi-edge graph and simple graph associated with $\pi$.
    We embed $G$ into $J(n, q+2, q+1)$ and denote by $\tilde{G}\subset J(n, q+2, q+1)$ as the image.
    Let $\tau_1,\cdots, \tau_r$ be  the vertices in $\tilde{G}$, 
    we can regard each vertex $\tau_i$ as a $(q+1)$-simplices and let $b_{\tau_i}$ be the associated boundary.
\begin{lemma}\label{lemma:S_pi_positive}
    Let $\#\loops(H_\pi)$ be the number of self-loops in $H_\pi$.
    Consider the term
    \[
    S_{\pi}(\tilde{G})=\prod_{i=1}^{2r}b_{\tau_{\pi(i)}}^\top b_{\tau_{\pi(i+1)}}.
    \]
  If $|\tau_1\cup\cdots \cup\tau_r|=q+r+1$,
  we have
    \begin{equation}\label{eq:S_pi_embedding}
    S_{\pi}(\tilde{G}) =(q+2)^{\#\loops(H_\pi)}.
    \end{equation}
\end{lemma}

\begin{proof}
We observe that
\[
b_{\tau_{\pi(i)}}^\top b_{\tau_{\pi(i+1)}}=
\begin{cases}
    q+2 & \textrm{ if } \tau_{\pi(i)}=\tau_{\pi(i+1)};\\
    \pm 1 & \textrm{ if } \tau_{\pi(i)} \textrm{ and } \tau_{\pi(i+1)} \textrm{ are upper adjacent};\\
    0 & \textrm{ else}.
\end{cases}
\]
Therefore,
each self-loop in $H_\pi$  contributes a factor $(q+2)$ to $S_{\pi}(\tilde{G})$,
and each edge between two different vertices in $H_\pi$ contributes a factor $\pm 1$ to $S_{\pi}(\tilde{G})$.  
Hence 
\[
S_\pi(\tilde{G})=\pm (q+2)^{\#\loops(H_\pi)},
\]
and we only need to prove $S_\pi(\tilde{G})$ is positive.

Observe that each edge between two different vertices $\tau_{\pi(i)}$ and $\tau_{\pi(i+1)}$ corresponds to either the term $\tau_{\pi(i)}^\top \tau_{\pi(i+1)}$ or the term $\tau_{\pi(i+1)}^\top \tau_{\pi(i)}$ in $S_\pi(\tilde{G})$.
Notice that
$\tau_{\pi(i)}^\top \tau_{\pi(i+1)}=\tau_{\pi(i+1)}^\top \tau_{\pi(i)}=\pm 1$,
that is,
the even number of mutli-edges between two different vertices in $H_\pi$ contribute to $1$ in $S_\pi(\tilde{G})$.
Hence,
we can modulo all the multi-edges by $2$, 
that is,
if there are even number multi-edges between $\tau_{\pi(i)}$ and $\tau_{\pi(i+1)}$,
we simply remove all those edges,
if there are odd number multi-edges between $\tau_{\pi(i)}$ and $\tau_{\pi(i+1)}$,
we only keep one single edge.
Recall that the mult-edge graph $H_\pi$ is 4-regular,
up to these operations,
the degree of each vertex in the remaining graph is still even.
Thus, the graph can be decomposed into edge-disjoint cycles (by Veblen's Theorem),
and we only need to consider each cycle.

However,
if $C$ is such a cycle that contains $u$ different  vertices 
\[\tau_{C_1} - \tau_{C_2} -\cdots - \tau_{C_u}-\tau_{C_1},\]
we must have $|\tau_{C_1}\cup\cdots \cup \tau_{C_u}|=q+u+1$,
otherwise,
it would be contrasted with the assumption that $|\tau_1\cup\cdots \cup\tau_r|=q+r+1$,
hence they must share a common $q$-face $\sigma$.
We can then denote each $(q+1)$-simplices as $\tau_{{C}_j}=\sigma\cup \{v_j\}$.
Let $c_{j}$ be the boundary coefficient of the common $q$-face $\sigma$ in the $(q+1)$-simplex $\tau_{C_{j}}$,
then 
\[
b_{\tau_{C_{j}}}^\top b_{\tau_{C_{j+1}}}=c_{j} c_{j+1},
\]
therefore,
in this cycle $C$,
we have
\[
\prod_{j=1}^u b_{\tau_{C_{j}}}^\top b_{\tau_{C_{j+1}}} = \prod_{j=1}^u c_{j} c_{j+1}=
\prod_{j=1}^u (c_{j})^2=1.
\]
That is,
each cycle contributes to $1$ in $S_\pi(\tilde{G})$,
and therefore,
\cref{eq:S_pi_embedding} is proved.
\end{proof}

According to \cref{lemma:S_pi_positive},
the value of $S_\pi(\tilde{G})$ is independent of the embedding map $G_{\pi}\hookrightarrow J(n, q+2, q+1)$ provided that $|\tau_1\cup\cdots \cup\tau_r|=q+r+1$,
therefore,
we can simplify the notation and denote by
\[
S_\pi := (q+2)^{\#\loops(H_\pi)}.
\]
Let 
\[
E_\pi(n)=\#\left\{\varphi: G_\pi\hookrightarrow J(n, q+2, q+1) \textrm{ is an embedding } \Big| |\varphi(1)\cup\varphi(2)\cdots \cup \varphi(r)|=q+r+1\right\},
\]
that is,
$E_\pi(n)$ is the number of embeddings of $G_\pi$ into $J(n, q+2, q+1)$ requiring that the support of the $r$ vertices in $J(n, q+2, q+1)$ is equal to $q+r+1$.
The \cref{eq:tilde_moment_orig} can be simplified as
\begin{equation}\label{eq:tilde_moment_simplify}
m_{2r}^{(n)}\sim
\frac{(\eta_{2})^r }{N_qn^r}\sum_{\pi\in \mathcal{P}_2(2r)}S_\pi E_\pi(n).
\end{equation}

\begin{lemma}\label{lemma:even_moments_limit}
We have
\begin{equation}\label{eq:moment_limit}
\lim_{n\to\infty}m_{2r}^{(n)}=(\eta_2)^r
\frac{(q+1)!}{(q+r+1)!}\sum_{\pi\in \mathcal{P}_2(2r)}S_\pi E_\pi(q+r+1).
\end{equation}
\end{lemma}

\begin{proof}
    Recall that the second centralized moment of $\Ber(p)$ is $\eta_{2}=p(1-p)$,
    observe that 
    \[
    E_\pi(n)=\binom{n}{q+r+1}E_\pi(q+r+1),
    \]
    and
    \[
    N_q=\binom{n-1}{q+1}\sim \frac{n^{q+1}}{(q+1)!}, \quad 
    \binom{n}{q+r+1}\sim\frac{n^{q+r+1}}{(q+r+1)!},
    \]
    hence \cref{eq:moment_limit} is proved from \cref{eq:tilde_moment_simplify} and \cref{lemma:even_moments}.
\end{proof}

\begin{definition}
    We say $(q+1)$-simplices $\tau_1,\cdots, \tau_k$ have a \emph{sunflower} structure if they share a common $q$-face $\sigma$,
    that is,
    $\sigma=\tau_1\cap \cdots \cap \tau_k$.
\end{definition}

\begin{definition}
    Let $G$ be a simple connected graph,
    that is, $G$ does not have self-loops and multi-edges.
    We say a vertex $v$ is an \emph{articulation vertex} if $G$ becomes disconnected by removing $v$ and the edges connected with $v$. 
\end{definition}

\begin{lemma}\label{lemma:articulation}
    Let $G_\pi$ be the simple graph associated with $\pi\in \mathcal{P}_2(2r)$.
    Let $\#\art(G_\pi)$ be the number of articulation vertices in $G_\pi$.
    We have the equality
    \[
    \frac{(q+1)!}{(q+r+1)!}E_\pi(q+r+1)=(q+2)^{\#\art(G_\pi)}.
    \]
\end{lemma}

\begin{proof}
Let $\tilde{G}_\pi\subset J(q+r+1, q+2, q+1)$ be an embedding image of $G_\pi$.
There are $r$ vertices in $\tilde{G}_\pi$,
we denote them as $\{\tau_1,\cdots, \tau_r\}$,
and we regard each vertex $\tau_i$ as a $(q+1)$-simplex, and two vertices $\tau_i$, $\tau_j$ are adjacent in $\tilde{G}_\pi$ if and only if they are upper adjacent in $\LL_n$.

Since there are only $q+r+1$ elements in the whole underlying space,
every time adding one new element must creating precisely one new $(q+1)$-simplices,
hence the structure of $\tilde{G}_\pi$ must be some sunflower structures connected by some articulation vertices,
and therefore,
only the articulation vertices in $G_\pi$ have the freedom, since each $(q+1)$-simplex has $(q+2)$ different $q$-simplices,
the freedom is $(q+2)$,
hence once the first vertex $\tau_1$ and the edge $\tau_1\sim \tau_2$ are fixed,
there are $(q+2)^{\#\art(G_\pi)}$ options,
we then only need to consider how many options to choose $\tau_1$ and the edge $\tau_1\sim \tau_2$.

We first choose the vertex $\tau_1$.
There are $\binom{q+r+1}{q+2}$ different options,
then we need to choose a $q$-face of $\tau_1$ which is also a $q$-face of $\tau_2$, 
since there are $(q+2)$ different $q$-faces,
the options are $\binom{q+2}{1}$.
once we fixed $\tau_1$,
there are $r-1$ vertices remained to be distributed,
there are $(r-1)!$ different methods,
hence in total 
\[
E_\pi(q+r+1)=
\binom{q+r+1}{q+2}\binom{q+2}{1}(r-1)!(q+2)^{\#\art(G_\pi)}=
\frac{(q+r+1)!}{(q+1)!}(q+2)^{\#\art(G_\pi)}.
\]

\end{proof}

\begin{example}\label{ex:flowers}
Let $G_\pi$ be a graph as in \cref{fig:flowers},
there are $8$ vertices in total,
by our definition, there are $2$ articulation vertices as colored in red.
Let $q=1$,
we embed the graph $G_\pi$ into $J(10, 3, 2)$,
since each articulation vertex has $3$ options to choose,
if we fix the initial $2$-simplex,
for example,
if we fix one articulate vertex in the Johnson graph,
there are in total $3^2$ possible embeddings,
we list all of them in \cref{fig:3x3_grid},
in each subfigure the red $2$-simplices correspond to the articulation vertices in $G_\pi$.
    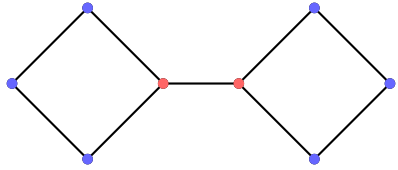
\begin{figure}[H]
    \centering
        \begin{tikzpicture}[thick]
            \draw (0,1) -- (-1,0) -- (0,-1) -- (1,0) -- cycle;
            
            \draw (3,1) -- (2,0) -- (3,-1) -- (4,0) -- cycle;
            
            \draw (1,0) -- (2,0);

            \foreach \x/\y in {0/1, -1/0, 0/-1, 1/0, 3/1, 2/0, 3/-1, 4/0}
                \fill (\x,\y) circle (2pt);
            \fill[blue!60] (0,1) circle (2pt);
            \fill[blue!60] (-1,0) circle (2pt);
            \fill[blue!60] (0,-1) circle (2pt);
            \fill[blue!60] (3,1) circle (2pt);
            \fill[blue!60] (3,-1) circle (2pt);
            \fill[blue!60] (4,0) circle (2pt);

            \fill[red!60] (1,0) circle (2pt);
            \fill[red!60] (2,0) circle (2pt);
        \end{tikzpicture}%
    \caption{The simple graph $G_\pi$ in \cref{ex:flowers}.}
    \label{fig:flowers}
\end{figure}
    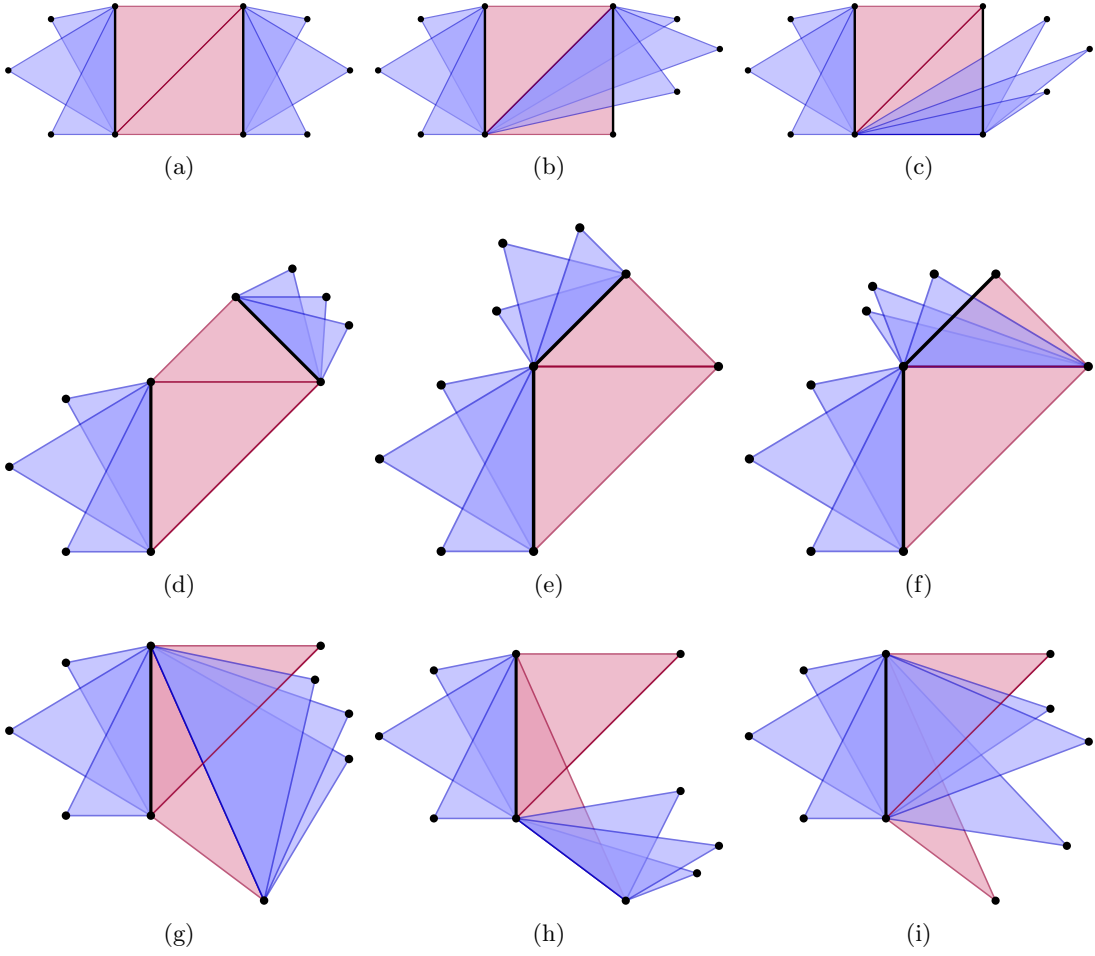
\begin{figure}[H]
    \centering
    \tikzset{
        vertex/.style={circle, fill=black, inner sep=1.5pt},
        hinge/.style={ultra thick, black, line cap=round},
        face1/.style={fill=blue!40, draw=blue!80!black, thick, opacity=0.55, line join=round},
        face2/.style={fill=red!40, draw=red!80!black, thick, opacity=0.55, line join=round},
        bridge/.style={fill=purple!40, draw=purple!80!black, thick, opacity=0.65, line join=round}
    }

    \begin{subfigure}[b]{0.32\textwidth}
        \centering
        \resizebox{\linewidth}{!}{%
            \begin{tikzpicture}
                \coordinate (H1_T) at (0, 1.5);
                \coordinate (H1_B) at (0, -1.5);
                \coordinate (H2_T) at (3, 1.5);
                \coordinate (H2_B) at (3, -1.5);

                \coordinate (P1_1) at (-1.5, 1.2);
                \coordinate (P1_2) at (-2.5, 0);
                \coordinate (P1_3) at (-1.5, -1.5);
                \coordinate (P2_1) at (4.5, 1.2);
                \coordinate (P2_2) at (5.5, 0);
                \coordinate (P2_3) at (4.5, -1.5);

                \filldraw[face1] (H1_T) -- (H1_B) -- (P1_1) -- cycle;
                \filldraw[face1] (H1_T) -- (H1_B) -- (P1_2) -- cycle;
                \filldraw[face1] (H1_T) -- (H1_B) -- (P1_3) -- cycle;
                \filldraw[bridge] (H1_T) -- (H1_B) -- (H2_T) -- cycle;
                \filldraw[bridge] (H2_T) -- (H2_B) -- (H1_B) -- cycle;
                \filldraw[face1] (H2_T) -- (H2_B) -- (P2_1) -- cycle;
                \filldraw[face1] (H2_T) -- (H2_B) -- (P2_2) -- cycle;
                \filldraw[face1] (H2_T) -- (H2_B) -- (P2_3) -- cycle;

                \draw[hinge] (H1_T) -- (H1_B);
                \draw[hinge] (H2_T) -- (H2_B);
                \draw[thick, purple!80!black, opacity=0.8] (H1_B) -- (H2_T);

                \node[vertex] at (H1_T) {};
                \node[vertex] at (H1_B) {};
                \node[vertex] at (H2_T) {};
                \node[vertex] at (H2_B) {};
                \node[vertex] at (P1_1) {};
                \node[vertex] at (P1_2) {};
                \node[vertex] at (P1_3) {};
                \node[vertex] at (P2_1) {};
                \node[vertex] at (P2_2) {};
                \node[vertex] at (P2_3) {};
            \end{tikzpicture}%
        }
         \caption{}
    \end{subfigure}\hfill
    \begin{subfigure}[b]{0.32\textwidth}
        \centering
        \resizebox{\linewidth}{!}{%
            \begin{tikzpicture}
                \coordinate (H1_T) at (0, 1.5);
                \coordinate (H1_B) at (0, -1.5);
                \coordinate (H2_T) at (3, 1.5);
                \coordinate (H2_B) at (3, -1.5);

                \coordinate (P1_1) at (-1.5, 1.2);
                \coordinate (P1_2) at (-2.5, 0);
                \coordinate (P1_3) at (-1.5, -1.5);
                \coordinate (P2_1) at (4.5, 1.2);
                \coordinate (P2_2) at (5.5, 0.5);
                \coordinate (P2_3) at (4.5, -0.5);

                \filldraw[face1] (H1_T) -- (H1_B) -- (P1_1) -- cycle;
                \filldraw[face1] (H1_T) -- (H1_B) -- (P1_2) -- cycle;
                \filldraw[face1] (H1_T) -- (H1_B) -- (P1_3) -- cycle;
                \filldraw[bridge] (H1_T) -- (H1_B) -- (H2_T) -- cycle;
                \filldraw[bridge] (H2_T) -- (H2_B) -- (H1_B) -- cycle;
                \filldraw[face1] (H2_T) -- (H1_B) -- (P2_1) -- cycle;
                \filldraw[face1] (H2_T) -- (H1_B) -- (P2_2) -- cycle;
                \filldraw[face1] (H2_T) -- (H1_B) -- (P2_3) -- cycle;

                \draw[hinge] (H1_T) -- (H1_B);
                \draw[hinge] (H2_T) -- (H2_B);
                \draw[thick, purple!80!black, opacity=0.8] (H1_B) -- (H2_T);

                \node[vertex] at (H1_T) {};
                \node[vertex] at (H1_B) {};
                \node[vertex] at (H2_T) {};
                \node[vertex] at (H2_B) {};
                \node[vertex] at (P1_1) {};
                \node[vertex] at (P1_2) {};
                \node[vertex] at (P1_3) {};
                \node[vertex] at (P2_1) {};
                \node[vertex] at (P2_2) {};
                \node[vertex] at (P2_3) {};
            \end{tikzpicture}%
        }
         \caption{}
    \end{subfigure}\hfill
    \begin{subfigure}[b]{0.32\textwidth}
        \centering
        \resizebox{\linewidth}{!}{%
            \begin{tikzpicture}
                \coordinate (H1_T) at (0, 1.5);
                \coordinate (H1_B) at (0, -1.5);
                \coordinate (H2_T) at (3, 1.5);
                \coordinate (H2_B) at (3, -1.5);

                \coordinate (P1_1) at (-1.5, 1.2);
                \coordinate (P1_2) at (-2.5, 0);
                \coordinate (P1_3) at (-1.5, -1.5);
                \coordinate (P2_1) at (4.5, 1.2);
                \coordinate (P2_2) at (5.5, 0.5);
                \coordinate (P2_3) at (4.5, -0.5);

                \filldraw[face1] (H1_T) -- (H1_B) -- (P1_1) -- cycle;
                \filldraw[face1] (H1_T) -- (H1_B) -- (P1_2) -- cycle;
                \filldraw[face1] (H1_T) -- (H1_B) -- (P1_3) -- cycle;
                \filldraw[bridge] (H1_T) -- (H1_B) -- (H2_T) -- cycle;
                \filldraw[bridge] (H2_T) -- (H2_B) -- (H1_B) -- cycle;
                \filldraw[face1] (H2_B) -- (H1_B) -- (P2_1) -- cycle;
                \filldraw[face1] (H2_B) -- (H1_B) -- (P2_2) -- cycle;
                \filldraw[face1] (H2_B) -- (H1_B) -- (P2_3) -- cycle;

                \draw[hinge] (H1_T) -- (H1_B);
                \draw[hinge] (H2_T) -- (H2_B);
                \draw[thick, purple!80!black, opacity=0.8] (H1_B) -- (H2_T);

                \node[vertex] at (H1_T) {};
                \node[vertex] at (H1_B) {};
                \node[vertex] at (H2_T) {};
                \node[vertex] at (H2_B) {};
                \node[vertex] at (P1_1) {};
                \node[vertex] at (P1_2) {};
                \node[vertex] at (P1_3) {};
                \node[vertex] at (P2_1) {};
                \node[vertex] at (P2_2) {};
                \node[vertex] at (P2_3) {};
            \end{tikzpicture}%
        }
        \caption{}
    \end{subfigure}

    \vspace{0.5cm} 

    \begin{subfigure}[b]{0.32\textwidth}
        \centering
        \resizebox{\linewidth}{!}{%
            \begin{tikzpicture}
                \coordinate (H1_T) at (0, 1.5);
                \coordinate (H1_B) at (0, -1.5);
                \coordinate (H2_T) at (3, 1.5);
                \coordinate (H2_B) at (1.5, 3);

                \coordinate (P1_1) at (-1.5, 1.2);
                \coordinate (P1_2) at (-2.5, 0);
                \coordinate (P1_3) at (-1.5, -1.5);
                \coordinate (P2_1) at (2.5, 3.5);
                \coordinate (P2_2) at (3.1, 3.0);
                \coordinate (P2_3) at (3.5, 2.5);

                \filldraw[face1] (H1_T) -- (H1_B) -- (P1_1) -- cycle;
                \filldraw[face1] (H1_T) -- (H1_B) -- (P1_2) -- cycle;
                \filldraw[face1] (H1_T) -- (H1_B) -- (P1_3) -- cycle;
                \filldraw[bridge] (H1_T) -- (H1_B) -- (H2_T) -- cycle;
                \filldraw[bridge] (H2_T) -- (H2_B) -- (H1_T) -- cycle;
                \filldraw[face1] (H2_T) -- (H2_B) -- (P2_1) -- cycle;
                \filldraw[face1] (H2_T) -- (H2_B) -- (P2_2) -- cycle;
                \filldraw[face1] (H2_T) -- (H2_B) -- (P2_3) -- cycle;

                \draw[hinge] (H1_T) -- (H1_B);
                \draw[hinge] (H2_T) -- (H2_B);
                \draw[thick, purple!80!black, opacity=0.8] (H1_B) -- (H2_T);

                \node[vertex] at (H1_T) {};
                \node[vertex] at (H1_B) {};
                \node[vertex] at (H2_T) {};
                \node[vertex] at (H2_B) {};
                \node[vertex] at (P1_1) {};
                \node[vertex] at (P1_2) {};
                \node[vertex] at (P1_3) {};
                \node[vertex] at (P2_1) {};
                \node[vertex] at (P2_2) {};
                \node[vertex] at (P2_3) {};
            \end{tikzpicture}%
        }
        \caption{}
    \end{subfigure}\hfill
    \begin{subfigure}[b]{0.32\textwidth}
        \centering
        \resizebox{\linewidth}{!}{%
            \begin{tikzpicture}
                \coordinate (H1_T) at (0, 1.5);
                \coordinate (H1_B) at (0, -1.5);
                \coordinate (H2_T) at (3, 1.5);
                \coordinate (H2_B_new) at (1.5, 3.0);

                \coordinate (P1_1) at (-1.5, 1.2);
                \coordinate (P1_2) at (-2.5, 0);
                \coordinate (P1_3) at (-1.5, -1.5);
                \coordinate (P2_1_new) at (-0.6, 2.4);
                \coordinate (P2_2_new) at (-0.5, 3.5);
                \coordinate (P2_3_new) at (0.75, 3.75);

                \filldraw[face1] (H1_T) -- (H1_B) -- (P1_1) -- cycle;
                \filldraw[face1] (H1_T) -- (H1_B) -- (P1_2) -- cycle;
                \filldraw[face1] (H1_T) -- (H1_B) -- (P1_3) -- cycle;
                \filldraw[bridge] (H1_T) -- (H1_B) -- (H2_T) -- cycle;
                \filldraw[bridge] (H1_T) -- (H2_B_new) -- (H2_T) -- cycle;
                \filldraw[face1] (H1_T) -- (H2_B_new) -- (P2_1_new) -- cycle;
                \filldraw[face1] (H1_T) -- (H2_B_new) -- (P2_2_new) -- cycle;
                \filldraw[face1] (H1_T) -- (H2_B_new) -- (P2_3_new) -- cycle;

                \draw[hinge] (H1_T) -- (H1_B);
                \draw[hinge] (H1_T) -- (H2_B_new);
                \draw[thick, purple!80!black, opacity=0.8] (H1_T) -- (H2_T);

                \node[vertex] at (H1_T) {};
                \node[vertex] at (H1_B) {};
                \node[vertex] at (H2_T) {};
                \node[vertex] at (H2_B_new) {};
                \node[vertex] at (P1_1) {};
                \node[vertex] at (P1_2) {};
                \node[vertex] at (P1_3) {};
                \node[vertex] at (P2_1_new) {};
                \node[vertex] at (P2_2_new) {};
                \node[vertex] at (P2_3_new) {};
            \end{tikzpicture}%
        }
        \caption{}
        \label{fig:sub5}
    \end{subfigure}\hfill
    \begin{subfigure}[b]{0.32\textwidth}
        \centering
        \resizebox{\linewidth}{!}{%
            \begin{tikzpicture}
                \coordinate (H1_T) at (0, 1.5);
                \coordinate (H1_B) at (0, -1.5);
                \coordinate (H2_T) at (3, 1.5);
                \coordinate (H2_B_new) at (1.5, 3.0);

                \coordinate (P1_1) at (-1.5, 1.2);
                \coordinate (P1_2) at (-2.5, 0);
                \coordinate (P1_3) at (-1.5, -1.5);
                \coordinate (P2_1_new) at (-0.6, 2.4);
                \coordinate (P2_2_new) at (-0.5, 2.8);
                \coordinate (P2_3_new) at (0.5, 3);

                \filldraw[face1] (H1_T) -- (H1_B) -- (P1_1) -- cycle;
                \filldraw[face1] (H1_T) -- (H1_B) -- (P1_2) -- cycle;
                \filldraw[face1] (H1_T) -- (H1_B) -- (P1_3) -- cycle;
                \filldraw[bridge] (H1_T) -- (H1_B) -- (H2_T) -- cycle;
                \filldraw[bridge] (H1_T) -- (H2_B_new) -- (H2_T) -- cycle;
                \filldraw[face1] (H1_T) -- (H2_T) -- (P2_1_new) -- cycle;
                \filldraw[face1] (H1_T) -- (H2_T) -- (P2_2_new) -- cycle;
                \filldraw[face1] (H1_T) -- (H2_T) -- (P2_3_new) -- cycle;

                \draw[hinge] (H1_T) -- (H1_B);
                \draw[hinge] (H1_T) -- (H2_B_new);
                \draw[thick, purple!80!black, opacity=0.8] (H1_T) -- (H2_T);

                \node[vertex] at (H1_T) {};
                \node[vertex] at (H1_B) {};
                \node[vertex] at (H2_T) {};
                \node[vertex] at (H2_B_new) {};
                \node[vertex] at (P1_1) {};
                \node[vertex] at (P1_2) {};
                \node[vertex] at (P1_3) {};
                \node[vertex] at (P2_1_new) {};
                \node[vertex] at (P2_2_new) {};
                \node[vertex] at (P2_3_new) {};
            \end{tikzpicture}%
        }
        \caption{}
        \label{fig:sub6}
    \end{subfigure}

    \vspace{0.5cm} 

    \begin{subfigure}[b]{0.32\textwidth}
        \centering
        \resizebox{\linewidth}{!}{%
            \begin{tikzpicture}
                \coordinate (H1_T) at (0, 1.5);
                \coordinate (H1_B) at (0, -1.5);
                \coordinate (H2_T) at (3, 1.5);
                \coordinate (H2_B) at (2., -3);

                \coordinate (P1_1) at (-1.5, 1.2);
                \coordinate (P1_2) at (-2.5, 0);
                \coordinate (P1_3) at (-1.5, -1.5);
                \coordinate (P2_1) at (3.5, -0.5);
                \coordinate (P2_2) at (3.5, 0.3);
                \coordinate (P2_3) at (2.9, 0.9);

                \filldraw[face1] (H1_T) -- (H1_B) -- (P1_1) -- cycle;
                \filldraw[face1] (H1_T) -- (H1_B) -- (P1_2) -- cycle;
                \filldraw[face1] (H1_T) -- (H1_B) -- (P1_3) -- cycle;
                \filldraw[bridge] (H1_T) -- (H1_B) -- (H2_T) -- cycle;
                \filldraw[bridge] (H1_T) -- (H1_B) -- (H2_B) -- cycle;
                \filldraw[face1] (H1_T) -- (H2_B) -- (P2_1) -- cycle;
                \filldraw[face1] (H1_T) -- (H2_B) -- (P2_2) -- cycle;
                \filldraw[face1] (H1_T) -- (H2_B) -- (P2_3) -- cycle;

                \draw[hinge] (H1_T) -- (H1_B);
                \draw[thick, purple!80!black, opacity=0.8] (H1_B) -- (H2_T);

                \node[vertex] at (H1_T) {};
                \node[vertex] at (H1_B) {};
                \node[vertex] at (H2_T) {};
                \node[vertex] at (H2_B) {};
                \node[vertex] at (P1_1) {};
                \node[vertex] at (P1_2) {};
                \node[vertex] at (P1_3) {};
                \node[vertex] at (P2_1) {};
                \node[vertex] at (P2_2) {};
                \node[vertex] at (P2_3) {};
            \end{tikzpicture}%
        }
        \caption{}
        \label{fig:sub7}
    \end{subfigure}\hfill
    \begin{subfigure}[b]{0.32\textwidth}
        \centering
        \resizebox{\linewidth}{!}{%
            \begin{tikzpicture}
                \coordinate (H1_T) at (0, 1.5);
                \coordinate (H1_B) at (0, -1.5);
                \coordinate (H2_T) at (3, 1.5);
                \coordinate (H2_B) at (2., -3);

                \coordinate (P1_1) at (-1.5, 1.2);
                \coordinate (P1_2) at (-2.5, 0);
                \coordinate (P1_3) at (-1.5, -1.5);
                \coordinate (P2_1) at (3.3, -2.5);
                \coordinate (P2_2) at (3, -1);
                \coordinate (P2_3) at (3.7, -2);

                \filldraw[face1] (H1_T) -- (H1_B) -- (P1_1) -- cycle;
                \filldraw[face1] (H1_T) -- (H1_B) -- (P1_2) -- cycle;
                \filldraw[face1] (H1_T) -- (H1_B) -- (P1_3) -- cycle;
                \filldraw[bridge] (H1_T) -- (H1_B) -- (H2_T) -- cycle;
                \filldraw[bridge] (H1_T) -- (H1_B) -- (H2_B) -- cycle;
                \filldraw[face1] (H1_B) -- (H2_B) -- (P2_1) -- cycle;
                \filldraw[face1] (H1_B) -- (H2_B) -- (P2_2) -- cycle;
                \filldraw[face1] (H1_B) -- (H2_B) -- (P2_3) -- cycle;

                \draw[hinge] (H1_T) -- (H1_B);
                \draw[thick, purple!80!black, opacity=0.8] (H1_B) -- (H2_T);

                \node[vertex] at (H1_T) {};
                \node[vertex] at (H1_B) {};
                \node[vertex] at (H2_T) {};
                \node[vertex] at (H2_B) {};
                \node[vertex] at (P1_1) {};
                \node[vertex] at (P1_2) {};
                \node[vertex] at (P1_3) {};
                \node[vertex] at (P2_1) {};
                \node[vertex] at (P2_2) {};
                \node[vertex] at (P2_3) {};
            \end{tikzpicture}%
        }
        \caption{}
        \label{fig:sub8}
    \end{subfigure}\hfill
    \begin{subfigure}[b]{0.32\textwidth}
        \centering
        \resizebox{\linewidth}{!}{%
            \begin{tikzpicture}
                \coordinate (H1_T) at (0, 1.5);
                \coordinate (H1_B) at (0, -1.5);
                \coordinate (H2_T) at (3, 1.5);
                \coordinate (H2_B) at (2., -3);

                \coordinate (P1_1) at (-1.5, 1.2);
                \coordinate (P1_2) at (-2.5, 0);
                \coordinate (P1_3) at (-1.5, -1.5);
                \coordinate (P2_1) at (3.3, -2);
                \coordinate (P2_2) at (3, 0.5);
                \coordinate (P2_3) at (3.7, -0.1);

                \filldraw[face1] (H1_T) -- (H1_B) -- (P1_1) -- cycle;
                \filldraw[face1] (H1_T) -- (H1_B) -- (P1_2) -- cycle;
                \filldraw[face1] (H1_T) -- (H1_B) -- (P1_3) -- cycle;
                \filldraw[bridge] (H1_T) -- (H1_B) -- (H2_T) -- cycle;
                \filldraw[bridge] (H1_T) -- (H1_B) -- (H2_B) -- cycle;
                \filldraw[face1] (H1_B) -- (H1_T) -- (P2_1) -- cycle;
                \filldraw[face1] (H1_B) -- (H1_T) -- (P2_2) -- cycle;
                \filldraw[face1] (H1_B) -- (H1_T) -- (P2_3) -- cycle;

                \draw[hinge] (H1_T) -- (H1_B);
                \draw[thick, purple!80!black, opacity=0.8] (H1_B) -- (H2_T);

                \node[vertex] at (H1_T) {};
                \node[vertex] at (H1_B) {};
                \node[vertex] at (H2_T) {};
                \node[vertex] at (H2_B) {};
                \node[vertex] at (P1_1) {};
                \node[vertex] at (P1_2) {};
                \node[vertex] at (P1_3) {};
                \node[vertex] at (P2_1) {};
                \node[vertex] at (P2_2) {};
                \node[vertex] at (P2_3) {};
            \end{tikzpicture}%
        }
        \caption{}
        \label{fig:sub9}
    \end{subfigure}

    \caption{All 9 embedding simplicial complexes of $G_\pi$ into $J(10, 3, 2)$ in \cref{ex:flowers}.}
    \label{fig:3x3_grid}
\end{figure}
\end{example}

\begin{proof}[Proof of \cref{prop:moments}]
According to \cref{lemma:even_moments_limit} and \cref{lemma:articulation},
we have the following expression of the limit of even moments:
\begin{equation}\label{eq:loop_articulation}
\lim_{n\to \infty}m_{2r}^{(n)}
=
(\eta_2)^r\sum_{\pi\in \mathcal{P}_2(2r)}(q+2)^{\#\loops(H_\pi)+\#\art(G_\pi)}.
\end{equation}

We notice that an adjacent isolated chord in $\pi$ contributes to a loop in $H_\pi$,
and a non-adjacent isolated chord in $\pi$ contributes to an articulation vertex in $G_\pi$,
we then have 
\[
\#\ic(\pi)=\#\loops(H_\pi) + \#\art(G_\pi).
\]
hence we have the following equality:
\begin{equation}\label{eq:even_moments_isolate_chord}
    \lim_{n\to\infty}m_{2r}^{(n)}=
    (\eta_2)^r \sum\limits_{\pi\in \mathcal{P}_2(2r)}(q+2)^{\# \ic(\pi)},
\end{equation}
and according to \cref{thm:moments},
we have 
\[
 \lim_{n\to\infty}m_{2r}^{(n)}=m_{2r}.
\]
\end{proof}

\subsection{Variance of the moments}\label{sec:variance}

\begin{proposition}[Variance]\label{prop:variance}
For each $k\geq 1$ we have 
\begin{equation}\label{eq:variance}
\Var\left[\frac{1}{N_q}\tr(W_n^k)\right]
=O\left(\frac{1}{n^{q+1}}\right).
\end{equation}
\end{proposition}

\begin{proof}
Let 
\[\boldsymbol{\tau}=(\tau_1,\cdots, \tau_k)\in \Delta_{q+1}^k
\]
be a $k$-tuple of $(q+1)$-simplices,
and Let $S_{\boldsymbol{\tau}}$ be the set of distinct $(q+1)$-simplices in the tuple $\boldsymbol{\tau}$.
We define 
\[
\widehat{\xi}_{\boldsymbol{\tau}}:=\prod_{i=1}^k(\xi_{\tau_{i}}-p),
\]
and 
\[
\alpha(\boldsymbol{\tau}):=\tr\left(\prod_{i=1}^kb_{\tau_i}b_{\tau_i}^\top\right).
\]

we express the normalized trace as
\begin{equation}\label{eq:covariance}
\Var\left[\frac{1}{N_q}(\tr(W_n^k))\right]
=
\frac{1}{N_q^2 n^{k}}\sum_{\boldsymbol{\tau},\boldsymbol{\sigma}\in \Delta_{q+1}^k}\alpha(\boldsymbol{\tau})\alpha(\boldsymbol{\sigma})\times\Cov[\widehat{\xi}_{\boldsymbol{\tau}}, \widehat{\xi}_{\boldsymbol{\sigma}}].
\end{equation}
We observe that for each term
\[
\alpha(\boldsymbol{\tau})\alpha(\boldsymbol{\sigma})\times\Cov[\widehat{\xi}_{\boldsymbol{\tau}}, \widehat{\xi}_{\boldsymbol{\sigma}}]=O(1),
\]
hence we only need to count how many combinations of $\boldsymbol{\tau}, \boldsymbol{\sigma}\in \Delta_{q+1}^{k}$ can contribute to the sum in \cref{eq:covariance}.
Now suppose 
\[
\alpha(\boldsymbol{\tau})\alpha(\boldsymbol{\sigma})\times\Cov[\widehat{\xi}_{\boldsymbol{\tau}}, \widehat{\xi}_{\boldsymbol{\sigma}}]\neq 0.
\]
Then $\alpha(\boldsymbol{\tau})\neq 0$ implies that 
\[
\tau_1-\tau_2-\cdots -\tau_k-\tau_1
\]
forms a closed walk in the Johnson graph $J(n, q+2, q+1)$,
and similar to $\alpha(\boldsymbol{\sigma})$.

On the other hand,
by the definition of covariance
\[\Cov[\widehat{\xi}_{\boldsymbol{\tau}}, \widehat{\xi}_{\boldsymbol{\sigma}}]
=\EE[\widehat{\xi}_{\boldsymbol{\tau}} \widehat{\xi}_{\boldsymbol{\sigma}}] - \EE[\widehat{\xi}_{\boldsymbol{\tau}}]\times \EE[\widehat{\xi}_{\boldsymbol{\sigma}}],
\]
$\Cov[\widehat{\xi}_{\boldsymbol{\tau}}, \widehat{\xi}_{\boldsymbol{\sigma}}]\neq 0$
implies that $S_{\boldsymbol{\tau}}\cap S_{\boldsymbol{\sigma}}\neq\emptyset$,
if not,
we would have $\EE[\widehat{\xi}_{\boldsymbol{\tau}} \widehat{\xi}_{\boldsymbol{\sigma}}] = \EE[\widehat{\xi}_{\boldsymbol{\tau}}]\times \EE[\widehat{\xi}_{\boldsymbol{\sigma}}]$,
this is contrasted with our requirement that $\Cov[\widehat{\xi}_{\boldsymbol{\tau}}, \widehat{\xi}_{\boldsymbol{\sigma}}]\neq 0$. 
More than that,
we must have that each distinct $(q+1)$-simplex occurs at least twice in the total $2k$ positions.
If not,
suppose for some $(q+1)$-simplex $\rho\in S_{\boldsymbol{\tau}}\cup S_{\boldsymbol{\sigma}}$ occurs exactly once.
Assume without loss of generality,
that $\rho\in S_{\boldsymbol{\tau}}$,
then $\EE[\widehat{\xi}_{\boldsymbol{\tau}} \widehat{\xi}_{\boldsymbol{\sigma}}]= \EE[\widehat{\xi}_{\boldsymbol{\tau}}]=0$,
this is contradicted with our requirement.

Hence,
let $u=|S_{\boldsymbol{\tau}}\cup S_{\boldsymbol{\sigma}}|$ denote the number of distinct $(q+1)$-simplices occurring in the union of $S_{\boldsymbol{\tau}}$ and $S_{\boldsymbol{\sigma}}$.
We must have that 
\[
u\leq k.
\]
Moreover,
since $S_{\boldsymbol{\tau}}$ and $S_{\boldsymbol{\sigma}}$ are connected,
and $S_{\boldsymbol{\tau}}\cap S_{\boldsymbol{\sigma}}\neq \emptyset$,
we must have the set $S_{\boldsymbol{\tau}}\cup S_{\boldsymbol{\sigma}}$ in the Johnson graph $J(n, q+2, q+1)$ are connected.

We now choose a rooted spanning tree of these $u$ distinct vertices,
there are $\binom{n}{q+2}$ options and every remaining tree vertex is an upper neighbor of an already chosen parent,
hence there are $O(n)$ options,
in total,
there are 
\[
\binom{n}{q+2}\times O(n^{u-1})=O(n^{q+u+1})
\]
options to fix the $u$ distinct simplices.

We then need to assign them into the $2k$ positions:
\[
\tau_1\cdots, \tau_k, \sigma_1,\cdots, \sigma_k,
\]
there are at most $u^{2k}\leq k^{2k}=O(1)$ assignments.
Therefore,
\[
\Var\left[\frac{1}{N_q}(\tr(W_n^k))\right]
=
\frac{1}{N_q^2 n^{k}}O(n^{q+u+1})=O(\frac{1}{n^{q+k-u+1}}),
\]
finally since $u\leq k$,
the \cref{eq:variance} is proved.
\end{proof}

\subsection{Proof of Theorem \ref{thm:main}}\label{sec:main_thm_proof}

Before giving the final proof of \cref{thm:main},
we still need to show that the probability measure $\mu_{\lap}$ is \emph{determined by moments},
i.e.,
    if $\nu$ is another probability measure having the same moments,
    then $\nu=\mu_{\lap}$.

\begin{lemma}
  The probability measure 
    \[
    \mu_{\lap}=
    \mathcal N\bigl(0,\eta_2\bigr)
\boxplus
\operatorname{SC}\bigl((q+1)\eta_2\bigr)
    \] is determined by moments.  
\end{lemma}

\begin{proof}
    Since
    \[
    m_{2r}=\eta_2^r\sum_{\pi\in \mathcal{P}_2(2r)}(q+2)^{\#\ic(\pi)},
    \]
    we obtain that
    \[
    m_{2r} \leq \eta_2^r (q+2)^r(2r-1)!!,
    \]
    and since 
    \[
    (2r-1)!! \leq (2r)^r,
    \]
    thus
    \[
    m_{2r}^{-\frac{1}{2r}} \geq \frac{1}{\sqrt{2\eta_2 (q+2)}}r^{-\frac{1}{2}}.
    \]
    Consequently,
    $\sum_{r=1}^\infty m_{2r}^{-\frac{1}{2r}}$ is divergent,
    hence Carleman's condition holds and the probability measure $\mu_{\lap}$ is determined by moments.
\end{proof}

\begin{proof}[Proof of \cref{thm:main}]
When $q=0$,
the result follows from \cite{MR2759729}.
Hence we assume $q \geq 1$.
According to Chebyshev's inequality,
for fixed $k\geq 1$ and $\epsilon > 0$,
\[
\Pr\left(\left| \frac{1}{N_q}\tr(W_n^k)-m_k^{(n)}\right| \geq \epsilon\right) \leq \frac{\Var\left[\frac{1}{N_q}\tr(W_n^k)\right]}{\epsilon^2}=O\left(\frac{1}{n^{q+1}}\right),
\]
hence for $q\geq 1$
\[
\sum_{n} \Pr\left(\left| \frac{1}{N_q}\tr(W_n^k)-m_k^{(n)}\right| \geq \epsilon\right) < \infty,
\]
and according to Borel–Cantelli lemma,
we have
\[
\frac{1}{N_q}\tr(W_n^k) - m_k^{(n)}  \xrightarrow{n\to \infty} 0 \quad \textrm{almost surely}.
\]
In addition to that,
since 
\[
\lim_{n\to\infty}m_k^{(n)}=m_k,
\]
hence 
\[
\frac{1}{N_q}\tr(W_n^k) \xrightarrow{n\to \infty} m_k \quad \textrm{almost surely}.
\]
Let $\Omega_k$ be the event when $\frac{1}{N_q}\tr(W_n^k)$ converges to $m_k$,
then $\Pr(\Omega_k)=1$.
Let $\Omega=\bigcap_{k\geq 1} \Omega_k$.
Since $\Omega$ is a countable intersection,
hence $\Pr(\Omega)=1$.
On this event $\Omega$,
applying the Fréchet–Shohat theorem (the fundamental theorem of the moment method),
if all the moments of a sequence of probability measures converge to the moments of a probability measure which is determined by moments,
then the sequence of measures converges weakly to that measure.
Therefore,
\cref{eq:free_conv} is proved.
\end{proof}

\bibliography{sample}

\end{document}